\documentclass[onefignum,onetabnum]{siamart171218}

\usepackage{color}
\usepackage{amssymb}

\usepackage{lipsum}
\usepackage{amsfonts}
\usepackage{graphicx}
\usepackage{epstopdf}
\usepackage{algorithmic}
\usepackage{subfigure}
\usepackage{enumitem}
\ifpdf
  \DeclareGraphicsExtensions{.eps,.pdf,.png,.jpg}
\else
  \DeclareGraphicsExtensions{.eps}
\fi

\newsiamremark{remark}{Remark}
\newsiamremark{hypothesis}{Hypothesis}
\crefname{hypothesis}{Hypothesis}{Hypotheses}
\newsiamthm{claim}{Claim}

\headers{A positive and moment-preserving spectral method}{Z. Cai, B. Lin and M. Lin}

\title{A positive and moment-preserving Fourier spectral method\thanks{Submitted to the editors DATE.
\funding{Zhenning Cai was funded by the Academic Research Fund of the Ministry of Education of Singapore under grant No.~A-0004592-00-00. Meixia Lin was partially supported by the Singapore University of Technology and Design under MOE Tier 1 Grant SKI 2021{\_}02{\_}08.
}}}

\author{Zhenning Cai\thanks{Department of Mathematics, National University of Singapore, Singapore 119076 
  (\email{matcz@nus.edu.sg}).}
\and Bo Lin\thanks{Corresponding author. Department of Mathematics, National University of Singapore, Singapore 119076 
	(\email{linbo@u.nus.edu}).}
\and Meixia Lin\thanks{Engineering Systems and Design, Singapore University of Technology and Design, Singapore 487372 (\email{meixia\_lin@sutd.edu.sg}).} 
}

\usepackage{amsopn}

\makeatletter
\newcommand*{\addFileDependency}[1]{
  \typeout{(#1)}
  \@addtofilelist{#1}
  \IfFileExists{#1}{}{\typeout{No file #1.}}
}
\makeatother

\newcommand*{\myexternaldocument}[1]{%
    \externaldocument{#1}%
    \addFileDependency{#1.tex}%
    \addFileDependency{#1.aux}%
}

\ifpdf
\hypersetup{
  pdftitle={A positive and moment-preserving spectral method},
  pdfauthor={Z. Cai, B. Lin, and M. Lin}
}
\fi

\myexternaldocument{ex_supplement}

\begin{document}

\maketitle

\begin{abstract}
    This paper presents a novel Fourier spectral method that utilizes optimization techniques to ensure the positivity and conservation of moments in the space of trigonometric polynomials. We rigorously analyze the accuracy of the new method and prove that it maintains spectral accuracy. To solve the optimization problem, we propose an efficient Newton solver that has quadratic convergence rate. Numerical examples are provided to demonstrate the high accuracy of the proposed method. Our method is also integrated into the spectral solver of the Boltzmann equation, showing the benefit of our approach in applications.
\end{abstract}



\section{Introduction}
In kinetic theories, the distribution functions are introduced to describe the number density of particles in the position-velocity space. These functions are non-negative everywhere, and the moments of these functions often correspond to fundamental physical quantities such as density, energy and electric charge, many of which are conservative by fundamental laws of nature. These properties are also reflected in kinetic equations such as the Boltzmann equation, the radiative transfer equation and the Vlasov equation. Numerically, researchers have also been making efforts to preserve these properties \cite{cai2018entropic,shen2022positive,Dreyer1987maxentropy,positivepn2010,levermore1996moment,siamdvm1988,sn1988}.

In general, it can be observed that for spectral methods, it is much harder to achieve positivity due to the oscillatory behavior of high-frequency basis functions, and sometimes the conservation laws are also lost when domain truncation is needed. In particular, for the Boltzmann equation, the Fourier spectral method has a much lower computational complexity for the quadratic collision operator compared with other methods. However, the Fourier spectral method does not preserve momentum and energy conservation, so that the equilibrium turns out to be a constant instead of the Maxwellian \cite{filbet2011analysis}. The momentum and energy conservation can be fixed by a post-processing after each time step \cite{gamba2009jcp,rey2022sinum}, but the positivity is still absent, which may affect the quality of the solution in long-time simulations. The positivity of the solution can be recovered by applying filters \cite{cai2018entropic,Pareschi2000filter}, but a positivity-preserving filter will reduce the convergence rate to second order, and it is unclear how to combine the filtering and the conservation fix: applying one after another will ruin the property achieved by the former.

In this work, we focus on the Fourier spectral method and propose a numerical strategy that preserves both the moments and the positivity on collocation points. The approach is based on an optimization problem with both equality and inequality constraints, and it will be shown that our approximation retains the spectral accuracy under mild conditions. Our method can be considered as an extension of a previous work in \cite{rey2022sinum}, where only the equality constraints are considered to preserve moments. We will then introduce an efficient numerical algorithm for the optimization problem, and integrate the method into the solver of the Boltzmann equation. Since the negative part of the distribution function can be controlled in the original Fourier spectral method \cite{gamba2018lagrange,filbet2011analysis,hu2021neg}, the spectral accuracy is also observed in the positivity-preserving Boltzmann solver. With positivity guaranteed, we can further combine the solver with an entropic scheme introduced in \cite{entropy2022sinum}, so as to achieve good quality in long-time simulations.

The paper is organized as follows. Our main results are in \cref{sec:main}, our detailed proof is in \cref{sec:optproblem,sec:bigproof,sec:cacatbound}, a practical algorithm is given in \cref{sec:algorithm}, numerical results are in \cref{sec:numerical}, and the conclusions follow in \cref{sec:conclusions}.

\section{Main results}
\label{sec:main}

\subsection{Positive and moment-preserving projection}
The major task of this paper is to find a positive and moment-preserving spectral approximation of a positive periodic  function $f \in L_p^2([-\pi, \pi]^d)$, $f \geq 0$ in the discrete function space
\begin{equation} \label{eq:Sn}
    \mathbb{S}^{N} = \text{span} \{ e^{i k \cdot x} \mid k \in \mathcal{N}^d \} \cap \mathbb{R},
\end{equation}
where
\begin{displaymath}
\mathcal{N} = \{-N, \cdots, N\}.
\end{displaymath}
Since requiring the function to be pointwisely positive may harm the spectral accuracy, here we consider
the approximation in a subset of $\mathbb{S}^N$ containing only functions that are positive on all the collocation points $x_k = 2\pi k / (2N+1)$, $k\in \mathcal{N}^d$. For simplicity, we define this subset as $\mathbb{S}_+^N$:
\begin{equation}
    \mathbb{S}^{N}_+ = \{ f_N \in \mathbb{S}^{N} \mid f_N(x_k) \geq 0, \, \forall x_k = 2\pi k / (2N+1), \, k \in \mathcal{N}^d \}.
\end{equation}
Meanwhile, using $\langle \varphi(x) \rangle$ to denote the integral of $\varphi(x)$ on $[-\pi,\pi]^d$, we want the moments $\boldsymbol{\rho}(f) := \langle \boldsymbol{m}(x) f(x) \rangle$ to be preserved during the approximation, where $\boldsymbol{m}(x)$ refers to a vector of $M$ linearly independent polynomials. In addition, one component of $\boldsymbol{m}(x)$ is required to be a constant, which implies the conservation of mass. For the Boltzmann equation, we may consider choosing $\boldsymbol{m}(x) = (1, x, x^2)^{\intercal}$ to preserve mass, momentum and energy.
Under these constraints, a natural idea to find the approximation is to solve the following convex optimization problem:
\begin{equation}\label{eq:optieq}
    \Pi_+^N f := \operatorname*{argmin}_{g \in \mathbb{S}^{N}_+ } \| g - f\|_2^2, \qquad \text{s.t. }  \boldsymbol{\rho}(g) = \boldsymbol{\rho}(f).
\end{equation}
This approach is similar to the method in \cite{rey2022sinum}, where the authors only considered the preservation of moments, so that $f$ is approximated by
\begin{equation}\label{eq:optargmin}
    \Pi^N f := \operatorname*{argmin}_{g \in \mathbb{S}^{N} } \| g - f\|_2^2, \qquad \text{s.t. }  \boldsymbol{\rho}(g) = \boldsymbol{\rho}(f).
\end{equation}
Compared to the solution of \cref{eq:optargmin}, our new $L^2$ projection \cref{eq:optieq} requires only an additional constraint on the positivity. 

Here we state our main result as \cref{thm:spconverge}; the proof is deferred to \cref{sec:optproblem,sec:bigproof,sec:cacatbound}. Throughout this paper, we will use the notation ``$A \lesssim B$'' to denote the inequality ``$A \leq CB$'' for a constant $C$ that can depend on the dimension $d$, the size of the domain $(2\pi)^d$, the polynomials $\boldsymbol{m}$.

\begin{definition}
We recall the interpolation operator $\mathcal{I}^N$ from $C((-\pi,\pi)^d)$ to $\mathbb{S}^N$ and projection operator $\mathcal{P}^N$ from $L^2((-\pi,\pi)^d)$ to $\mathbb{S}^N$, which are 
\begin{equation}\label{eq:interpolate}
    (\mathcal{I}^N g)(x_k) = g(x_k), \  \forall k \in \mathcal{N}^d; \qquad \mathcal{P}^N g = \sum_{k \in \mathcal{N}^d} \hat{g}_k e^{i k \cdot x},
\end{equation}
where $\hat{g}_k$ is the Fourier coefficients
\begin{equation}\label{eq:fouriertran}
    \hat{g}_k = \frac{1}{(2\pi)^d} \int_{[-\pi, \pi]^d} g(x) e^{-ik \cdot x} \mathrm{d}x.
\end{equation}
\end{definition}

\begin{theorem}\label{thm:spconverge}
  For non-negative $f \in L_p^2([-\pi, \pi]^d) \cap C((-\pi,\pi)^d)$, under conditions
  \begin{enumerate}[label=(H\arabic*)]
    \item \label{Hcon1} the component functions of $\boldsymbol{m}(x)$, denoted as $m_j(x)$ for $1 \leq j \leq M$, are linearly independent;
    \item \label{Hcon3} $\exists g>0$ such that $\boldsymbol{\rho}(g) = \boldsymbol{\rho}(f)$;
    \item \label{Hcon4} $\exists g \in \mathbb{S}^N$ such that $g(x_k)\!>\! 0$ for $k\! \in \! \mathcal{N}^d$ and $\frac{(2\pi)^d}{(2N+1)^d} \sum_{k \in \mathcal{N}^d} g(x_k) \mathcal{P}^N \boldsymbol{m}(x_k)\! =\! \boldsymbol{\rho} (f)$; 
  \end{enumerate}
 there exists a $N_0>0$, such that for $N \geq N_0$, $\Pi^N_+ f$ in \cref{eq:optieq} is well defined, and
  \begin{equation}\label{eq:spectral}
      \| f - \Pi^N_+ f\|_{2} \lesssim \left( \| f - \Pi^N f \|_{2} + \| f - \mathcal{I}^N f \|_{2} \right).
  \end{equation}
\end{theorem}

Based on the spectral convergence of interpolation error and the spectral convergence of $\| f - \Pi^N f\|_2$ in \cite{rey2022sinum}, we directly obtain a corollary of the above theorem.

\begin{corollary}\label{cor:spectral}
    Under the conditions \ref{Hcon1}--\ref{Hcon4}, if $f \geq 0$ and $f \in H_p^r([-\pi, \pi]^d)$ where $r > d/2$ is an integer, it holds that for sufficiently large $N$,
    \begin{equation}
      \| f - \Pi^N_+ f\|_{2} \lesssim\frac{1}{N^r} \| f \|_{H^r_p}.
  \end{equation}
\end{corollary}

\section{On the optimization problem}\label{sec:optproblem}
Before proving the error estimate, we will first study the well-posedness of the optimization problem \cref{eq:optieq} and reformulate it into a different form. The major obstacle in the analysis is that the exact solution of \cref{eq:optieq} cannot be written due to the inequality constraints. The purpose of this section is to provide a convenient form of the solution that will be used in the error analysis.

\subsection{Existence and uniqueness}\label{sec:wellpose}
We will now prove that \cref{eq:optieq} admits a unique solution so that the operator $\Pi^N_+$ is well defined. 

To show the existence of \cref{eq:optieq}, it is more convenient to rewrite it as a finite-dimensional quadratic programming with linear constraints. Making use of Parseval's theorem, the optimization problem in \cref{eq:optieq} can be rewritten as
\begin{equation}\label{eq:parseopt}
    \min_{\hat{g}_k} \ (2 \pi)^d \sum_{k \in \mathcal{N}^d} |\hat{g}_k - \hat{f}_k |^2, \qquad \text{s.t. } C\hat{\boldsymbol{g}} \geq \boldsymbol{0} \text{ and } (2 \pi)^d \sum_{k \in \mathcal{N}^d} \hat{\boldsymbol{m}}_k \overline{\hat{g}_k} = \boldsymbol{\rho}(f),
\end{equation}
where $\hat{g}_k$ , $\hat{f}_k$ and $\hat{\boldsymbol{m}}_k$ are, respectively, the Fourier coefficients of function $g(x)$, $f(x)$ and $\boldsymbol{m}(x)$ defined as in \cref{eq:fouriertran}; the matrix $C\in \mathbb{C}^{(2N+1)^d \times (2N+1)^d}$ denotes the inverse discrete Fourier transform such that the components of $C \hat{\boldsymbol{g}}$ are $g(x_k)$ for $k \in \mathcal{N}^d$, and ``$\geq \boldsymbol{0}$'' means each component is non-negative.

The problem \cref{eq:parseopt} is a finite-dimensional quadratic programming with linear constraints and a strictly convex objective function. Therefore, the existence and uniqueness of the solution is guaranteed provided that the feasible set
\begin{displaymath}
\{ g \in \mathbb{S}^N \mid g(x_k) \geq 0 \ \forall k \in \mathcal{N}^d, (2 \pi)^d \sum_{k \in \mathcal{N}^d} \hat{\boldsymbol{m}}_k \overline{\hat{g}_k} = \boldsymbol{\rho}(f)\}   
\end{displaymath}
is not empty. Nevertheless, we introduce a stronger condition \ref{Hcon4} to fulfill our later proof, and this condition could be viewed as a sufficient condition for \cref{eq:parseopt} having a unique solution. Moreover, we can mimic the above proof to get the existence and uniqueness of \cref{eq:optargmin} under the same condition \ref{Hcon4}. 

\subsection{Reformulation of the optimization problem}
Compared with the operator $\Pi^N f$ defined in \cref{eq:optargmin}, the explicit form of $\Pi^N_+ f$ cannot be written due to the inequality constraints. As a result, the analysis of $\Pi^N_+ f$ is significantly harder. In this section, we will focus on the representation of $\Pi^N_+ f_N^c$ with $f_N^c \in \mathbb{S}^N$ by assuming the knowledge of the active constraints of the optimization problem.

We decompose $\mathbb{S}^{N}$ in \cref{eq:Sn} into the direct sum of two subspaces:
\begin{equation}\label{eq:decompS}
    \mathbb{S}^{N} = \mathbb{M} \oplus \mathbb{M}^{\perp},
\end{equation}
where
\begin{equation}\label{eq:mspace}
    \mathbb{M} = \{ f_N \in \mathbb{S}^{N} \mid \boldsymbol{\rho}(f_N) = \boldsymbol{0} \}.
\end{equation}
Thus, for any $f_N^c \in \mathbb{S}^N$, it holds that
\begin{equation}\label{eq:ftom}
   \| \Pi^N_+ (f_N^c) - f_N^c \|_2^2 = \min_{h \in \mathbb{M}} \| h \|_2^2, \  \text{s.t. } h(x_k) + f_N^c(x_k) \geq 0, \  k \in \mathcal{N}^d.
\end{equation}
Consider a real orthonormal basis of $\mathbb{M}$ such that $\mathbb{M} = \text{span}\{ \psi_l(x) \}$, then the optimization problem in \cref{eq:ftom} is equivalent to
\begin{equation}\label{eq:morthobase}
    \min_{h_l} \sum_{l} \frac{1}{2} |h_l|^2, \  \text{s.t. } \sum_{l} h_l \psi_l(x_k) + f_N^c(x_k) \geq 0,\  k \in \mathcal{N}^d. 
\end{equation}
Similar to the idea of the active set method (e.g., see \cite[Chapter 16]{nocedal2006numerical}), these inequality constraints can be decomposed into active and inactive constraints. Precisely speaking, given the vector $\boldsymbol{h}^* = (h_l^*)$ that optimizes the objective function in \cref{eq:morthobase}, we can separate the index set for the constraints $\mathcal{N}^d$ into two disjoint subsets $\mathcal{N}^d = \mathcal{A} \cup \mathcal{R}$, such that
\begin{displaymath}
\sum_{l} h_l^* \psi_l(x_k) + f_N^c(x_k) = 0,\  \forall k \in \mathcal{A}, \qquad
\sum_{l} h_l^* \psi_l(x_k) + f_N^c(x_k) > 0,\  \forall k \in \mathcal{R}. 
\end{displaymath}
This indicates that $\mathcal{A}$ is the set of active constraints, and therefore \cref{eq:morthobase} is equivalent to
\begin{equation}\label{eq:optactive}
\begin{aligned}
    \min_{h_l} \sum_{l} \frac{1}{2} |h_l|^2, \  \text{s.t. } & \sum_{l} h_l \psi_l(x_k) + f_N^c(x_k) = 0,\  k \in \mathcal{A}, \\
    & \sum_{l} h_l \psi_l(x_k) + f_N^c(x_k) > 0,\  k \in \mathcal{R}.
\end{aligned}
\end{equation}
For simplification, we define $\boldsymbol{f}_{\mathcal{A}}$ and $\boldsymbol{f}_{\mathcal{R}}$ as vectors of $f_N^c(x_k)$ for $k \in \mathcal{A}$ and $k \in \mathcal{R}$, respectively, and thus the constraints in \cref{eq:optactive} can be represented as 
\begin{equation} \label{eq:pos}
C_{\mathcal{A}} \boldsymbol{h} + \boldsymbol{f}_{\mathcal{A}} = \boldsymbol{0}, \qquad C_{\mathcal{R}} \boldsymbol{h} + \boldsymbol{f}_{\mathcal{R}} > \boldsymbol{0},
\end{equation}
where the components of $C_{\mathcal{A}}$ and $C_{\mathcal{R}}$ are the values of basis functions $\psi_l(x_k)$. Assuming that the active set $\mathcal{A}$ is given, we can write the Karush–Kuhn–Tucker (KKT) conditions as
\begin{equation}\label{eq:qpkkt}
    \boldsymbol{h} - C_{\mathcal{A}}^{\intercal} \boldsymbol{\lambda} = \boldsymbol{0}, \quad  C_{\mathcal{A}} \boldsymbol{h} = -\boldsymbol{f}_{\mathcal{A}}, \quad C_{\mathcal{R}} \boldsymbol{h} + \boldsymbol{f}_{R} > \boldsymbol{0},\quad \boldsymbol{\lambda} \geq \boldsymbol{0},
\end{equation}
where $\boldsymbol{\lambda} \in \mathbb{R}^{|\mathcal{A}|}$ is the Lagrange multiplier. Furthermore, we can assume that $C_{\mathcal{A}}$ in \cref{eq:qpkkt} is of full row rank, since if $C_{\mathcal{A}}$ has linearly dependent rows, we can simply remove some constraints such that all the remaining constraints are independent. 
Under this assumption, the first two equations in \cref{eq:qpkkt} can already determine the vectors $\boldsymbol{h}$ and $\boldsymbol{\lambda}$ uniquely. The solutions are
\begin{equation}\label{eq:activesolution}
    \boldsymbol{h} = -C_{\mathcal{A}}^{\intercal} (C_{\mathcal{A}} C_{\mathcal{A}}^{\intercal})^{-1} \boldsymbol{f}_{\mathcal{A}}, \qquad \boldsymbol{\lambda} = - (C_{\mathcal{A}} C_{\mathcal{A}}^{\intercal})^{-1} \boldsymbol{f}_{\mathcal{A}}.
\end{equation}
The well-posedness of \cref{eq:morthobase} or \cref{eq:optactive} guarantees that the solutions \cref{eq:activesolution} automatically satisfy the last two conditions in \cref{eq:qpkkt}.

The solution of $\boldsymbol{h}$ in the form of \cref{eq:activesolution} will be used in the proofs in the next sections. However, it should be remarked here that this form cannot be used in the algorithms since the active set $\mathcal{A}$ is generally unknown. A practical algorithm to solve \cref{eq:optieq} will be given later in \cref{sec:algorithm}.

\section{Proof of \cref{thm:spconverge}}\label{sec:bigproof}
We will prove \cref{eq:spectral} by several inequalities sequentially. The main idea is summarized as follows.
\begin{itemize}
    \item The optimality of $\Pi_+^N f$ shows that 
    \begin{equation}\label{eq:fpiptofpi}
    \| f - \Pi^N_+ f\|_2 \leq \| f - \Pi^N_+ (\Pi^N f)\|_2 \leq \| f - \Pi^N f \|_2 + \| \Pi^N f - \Pi^N_+ (\Pi^N f) \|_2,
    \end{equation}
    since $\Pi^N f \in \mathbb{S}^N$ and $\boldsymbol{\rho}(\Pi^N f) = \boldsymbol{\rho}(f)$.
    \item By choosing $f_N^c = \Pi^N f$ in \cref{eq:ftom}, its solution $\boldsymbol{h}$ \cref{eq:activesolution} satisfies 
    \begin{equation}\label{eq:hsolbound}
    \| \Pi^N_+ (\Pi^N f) - \Pi^N f \|_2 = \| \boldsymbol{h} \|_2 \leq \| C_{\mathcal{A}}^{\intercal} \|_2 \| (C_{\mathcal{A}} C_{\mathcal{A}}^{\intercal})^{-1} \|_2 \| \boldsymbol{f}_{\mathcal{A}} \|_2.
    \end{equation}
    \item Further proof requires independent estimations of the three terms on the right-hand side of \cref{eq:hsolbound}. The results are
    \begin{gather} 
    \label{eq:pifboundifneg}
        \| \boldsymbol{f}_{\mathcal{A}} \|_2 \lesssim (2N+1)^{d/2} \|C_{\mathcal{A}}\|_2^2\|(C_{\mathcal{A}}C_{\mathcal{A}}^{\intercal})^{-1}\|_2 \| \mathcal{I}^N \left( f_N^{c-} \right) \|_2, \\
    \label{eq:pifMinuspipf}
        \| C_{\mathcal{A}}\|_2 \lesssim (2N+1)^{d/2}, \\
    \label{eq:cactboundtoshow}
        \| (C_{\mathcal{A}} C_{\mathcal{A}}^{\intercal} )^{-1}\|_2 \lesssim (2N+1)^{-d},
    \end{gather}
    where $f_N^{c-} = \max(-f_N^c, 0)$ and $f_N^c$ again refers to $\Pi^N f$.
    \item Finally, the term $\| \mathcal{I}^N \left( f_N^{c-} \right) \|_2$ on the right-hand side of \cref{eq:pifboundifneg} can be estimated by
    \begin{equation} \label{eq:If}
        \| \mathcal{I}^N \left( f_N^{c-} \right) \|_2 \lesssim \| f - f_N^c \|_2 + \| f - \mathcal{I}^N f \|_2.
    \end{equation}
\end{itemize}
It is clear that concatenating all the inequalities above will lead to the conclusion of \cref{thm:spconverge}. In this section, we will show the proofs of \cref{eq:pifMinuspipf,eq:pifboundifneg,eq:If} in \cref{sec:hlessfa}, \cref{sec:faboundfneg} and \cref{sec:If}, respectively. The proof of \cref{eq:cactboundtoshow} is more involved, and we will defer it to \cref{sec:cacatbound}.

\subsection{Proof of \cref{eq:pifboundifneg}}\label{sec:faboundfneg}
To distinguish the negative and the positive components of $\boldsymbol{f}_{\mathcal{A}}$, we decompose it as
\begin{equation}\label{eq:fadecomppm}
    \boldsymbol{f}_{\mathcal{A}} = \boldsymbol{f}_{\mathcal{A}}^+ - \boldsymbol{f}_{\mathcal{A}}^-,
\end{equation}
where $\boldsymbol{f}_{\mathcal{A}}^+ = \max(\boldsymbol{f}_{\mathcal{A}}, \boldsymbol{0})$ and $\boldsymbol{f}_{\mathcal{A}}^- = \max(-\boldsymbol{f}_{\mathcal{A}}, \boldsymbol{0})$. By definition, the negative part can be estimated by
\begin{equation} \label{eq:fA-}
    \| \boldsymbol{f}_{\mathcal{A}}^- \|_2^2 = \sum_{ k \in \mathcal{A} } \left( \min(f_N^c(x_k), 0) \right)^2 \leq \sum_{ k \in \mathcal{N}^d } \left( f_N^{c-}(x_k) \right)^2 = \frac{(2N+1)^d}{(2 \pi)^d} \| \mathcal{I}^N (f_N^{c-})\|_2^2,
\end{equation}
where the last equality comes from Parseval's theorem and $\mathcal{I}^N$ is the interpolation operator \cref{eq:interpolate}. Therefore, to prove \cref{eq:pifboundifneg}, we just need to use $\|\boldsymbol{f}_{\mathcal{A}}^-\|$ to control $\|\boldsymbol{f}_{\mathcal{A}}^+\|$, which will be done in the rest part of this subsection.

Since $C_{\mathcal{A}}$ has full row rank, its pseudoinverse $C_{\mathcal{A}}^{\dagger} = C_{\mathcal{A}}^{\intercal} (C_{\mathcal{A}} C_{\mathcal{A}}^{\intercal})^{-1}$. To make use of the optimality of $\boldsymbol{h}$ in \cref{eq:optactive}, we consider the following perturbation of $\boldsymbol{h}$:
\begin{equation}
    \boldsymbol{h}_{\varepsilon} = \boldsymbol{h} + \varepsilon C_{\mathcal{A}}^{\dagger} \boldsymbol{f}_{\mathcal{A}}^+,
\end{equation}
where $\varepsilon$ is chosen as a positive real number. Recalling the definition of $C_{\psi}$ in \cref{eq:decomposeC}, direct computation shows
\begin{equation}\label{eq:feasible}
\begin{split}
    C_{\psi} \boldsymbol{h}_{\varepsilon} + \boldsymbol{f} = C_{\psi} \boldsymbol{h} + \boldsymbol{f} + \varepsilon C_{\psi} C_{\mathcal{A}}^{\dagger} \boldsymbol{f}_{\mathcal{A}}^+ & =
    C_{\psi} \boldsymbol{h} + \boldsymbol{f} + \varepsilon \begin{pmatrix}
        C_{\mathcal{A}} \\
        C_{\mathcal{R}}
    \end{pmatrix}
    C_{\mathcal{A}}^{\intercal} (C_{\mathcal{A}} C_{\mathcal{A}}^{\intercal})^{-1}
    \boldsymbol{f}_{\mathcal{A}}^+ \\
    & = \begin{pmatrix}
    C_{\mathcal{A}} \boldsymbol{h} + \boldsymbol{f}_{\mathcal{A}} \\
    C_{\mathcal{R}} \boldsymbol{h} + \boldsymbol{f}_{\mathcal{R}}
    \end{pmatrix} + \varepsilon \begin{pmatrix}
        \boldsymbol{f}_{\mathcal{A}}^+ \\
        C_{\mathcal{R}} C_{\mathcal{A}}^{\dagger} \boldsymbol{f}_{\mathcal{A}}^+
    \end{pmatrix}.
\end{split}
\end{equation}
Since $C_{\mathcal{A}} \boldsymbol{h} + \boldsymbol{f}_{\mathcal{A}} = \boldsymbol{0}$ and $C_{\mathcal{R}} \boldsymbol{h} + \boldsymbol{f}_{\mathcal{R}} > \boldsymbol{0}$ (see \cref{eq:pos}), it can be seen that all components of the vector in the right-hand side of \cref{eq:feasible} are non-negative for sufficiently small $\varepsilon$. Therefore, $\boldsymbol{h}_{\varepsilon}$ stays in the feasible set, and the optimality of $\boldsymbol{h}$ suggests
\begin{equation}
    \frac{\mathrm{d}}{\mathrm{d} \varepsilon} \left( \frac{1}{2} \| \boldsymbol{h}_{\varepsilon} \|_2^2 \right) \bigg|_{\varepsilon = 0^+} \geq 0,
\end{equation}
which is
\begin{equation}
    0 \leq \boldsymbol{h}^{\intercal} C_{\mathcal{A}}^{\dagger} \boldsymbol{f}_{\mathcal{A}}^+ = 
    \boldsymbol{h}^{\intercal} C_{\mathcal{A}}^{\intercal} (C_{\mathcal{A}} C_{\mathcal{A}}^{\intercal})^{-1} \boldsymbol{f}_{\mathcal{A}}^+ = - \boldsymbol{f}_{\mathcal{A}}^{\intercal} (C_{\mathcal{A}} C_{\mathcal{A}}^{\intercal})^{-1} \boldsymbol{f}_{\mathcal{A}}^+.
\end{equation}
Then we plug \cref{eq:fadecomppm} into the above equation to get
\begin{displaymath}
    (\boldsymbol{f}_{\mathcal{A}}^-)^{\intercal} (C_{\mathcal{A}} C_{\mathcal{A}}^{\intercal})^{-1} \boldsymbol{f}_{\mathcal{A}}^+ - (\boldsymbol{f}_{\mathcal{A}}^+)^{\intercal} (C_{\mathcal{A}} C_{\mathcal{A}}^{\intercal})^{-1} \boldsymbol{f}_{\mathcal{A}}^+ \geq 0,
\end{displaymath}
which yields
    $
    \| C_{\mathcal{A}}^{\dagger} \boldsymbol{f}_{\mathcal{A}}^{+} \|_2^2 \leq \| C_{\mathcal{A}}^{\dagger} \boldsymbol{f}_{\mathcal{A}}^{+} \|_2 \| C_{\mathcal{A}}^{\dagger} \|_2 \| \boldsymbol{f}_{\mathcal{A}}^- \|_2
    $.
Therefore,
\begin{equation}
\begin{split}
    \| \boldsymbol{f}_{\mathcal{A}}^+ \|_2 
    \leq \| C_{\mathcal{A}} \|_2 \| C_{\mathcal{A}}^{\dagger} \boldsymbol{f}_{\mathcal{A}}^+ \|_2 
    & \leq \| C_{\mathcal{A}} \|_2 \| C_{\mathcal{A}}^{\dagger} \|_2 \| \boldsymbol{f}_{\mathcal{A}}^- \|_2 \\
    & \leq \| C_{\mathcal{A}} \|_2 \| C_{\mathcal{A}}^{\intercal} \|_2 \| (C_{\mathcal{A}} C_{\mathcal{A}}^{\intercal})^{-1}\|_2 \| \boldsymbol{f}_{\mathcal{A}}^- \|_2.
\end{split}
\end{equation}
Combining \cref{eq:fA-}  and the above equation yields \cref{eq:pifboundifneg}.

\subsection{Proof of \cref{eq:pifMinuspipf}}\label{sec:hlessfa}

To show the inequality \cref{eq:pifMinuspipf}, we recall that the entries of $C_{\mathcal{A}}$ are $\psi_l(x_k)$ with $k \in \mathcal{A}$ and $l$ being the index of the basis functions of $\mathbb{M}$. The following lemma is a basic property of Fourier spectral methods:
\begin{lemma} \label{lemma:ortho} For any two basis functions $\psi_j$ and $\psi_l$,
\begin{equation}
    \sum_{k \in \mathcal{N}^d} \psi_j(x_k) \psi_l(x_k) = \frac{(2N+1)^d}{(2\pi)^d} \delta_{jl}.
\end{equation}
\end{lemma}

This conclusion can be drawn by the fact that $\{\psi_l(x)\}$ is an orthonormal basis of a subspace of $\mathbb{S}^N$. We can now show the estimate \cref{eq:pifMinuspipf} as follows:

\begin{proof}[Proof of \cref{eq:pifMinuspipf}]
    Let
    \begin{equation}\label{eq:decomposeC}
    C_{\psi} = \begin{pmatrix}
    C_{\mathcal{A}} \\
    C_{\mathcal{R}}
    \end{pmatrix}.
    \end{equation}
    \cref{lemma:ortho} indicates that $C_{\psi}^{\intercal} C_{\psi} = \frac{(2N+1)^d}{(2\pi)^d} I$, where $I$ denotes the identity matrix. Then for any vector $\boldsymbol{y}$, it holds that
\begin{equation}\label{eq:CAboundnorm}
    \| C_{\psi} \boldsymbol{y}\|_2^2 = \boldsymbol{y}^{\intercal} C_{\psi}^{\intercal} C_{\psi} \boldsymbol{y} = \boldsymbol{y}^{\intercal} C_{\mathcal{A}}^{\intercal} C_{\mathcal{A}} \boldsymbol{y} + \boldsymbol{y}^{\intercal} C_{\mathcal{R}}^{\intercal} C_{\mathcal{R}} \boldsymbol{y} \geq \boldsymbol{y}^{\intercal} C_{\mathcal{A}}^{\intercal} C_{\mathcal{A}} \boldsymbol{y} = \| C_{\mathcal{A}} \boldsymbol{y}\|_2^2.
\end{equation}
Therefore, 
\begin{equation}\label{eq:catbound}
    \| C_{\mathcal{A}}^{\intercal} \|_2 = \| C_{\mathcal{A}} \|_2 \leq \| C_{\psi} \|_2 = \sqrt{\frac{(2N+1)^d}{(2\pi)^d}}.
\end{equation}
\end{proof}

\subsection{Proof of \cref{eq:If}} \label{sec:If}

Below we will prove a more general result in \cref{lemma:negbound}, where setting $g = f_N^c$ gives the inequality \cref{eq:If}.
\begin{lemma}\label{lemma:negbound}
    For a given non-negative continuous function $f \geq 0$, it holds that for any continuous function $g$,
    \begin{displaymath}
        \| \mathcal{I}^N (g^-)\|_2 \leq 2 \left( \| f - \mathcal{I}^N f \|_2  + \| f - \mathcal{I}^N g \|_2 \right)
    \end{displaymath}
\end{lemma}

\begin{proof}
    Making use of the discrete Fourier transform and Parseval's theorem, it is not difficult to see that for any two functions $g$ and $w$ in $\mathbb{S}^N$,
    \begin{equation}\label{eq:sninnerprod}
        \langle g\,w\rangle = (2 \pi)^d \sum_{k \in \mathcal{N}^d} \hat{g}_k \overline{\hat{w}_k} = \frac{(2 \pi)^d}{(2N+1)^d} \sum_{k \in \mathcal{N}^d} g(x_k) w(x_k).
    \end{equation}
    Using $f \geq 0$ and \cref{eq:sninnerprod}, it holds that $\| \mathcal{I}^N ( g^{+} - f ) \|_2 \leq \| \mathcal{I}^N (g - f ) \|_2$. Therefore,
    \begin{align*}
        \| \mathcal{I}^N (g^{-}) \|_2 & = \| \mathcal{I}^N (g^{-} - g^{+} + f - f + g^{+} ) \|_2 \\
        & \leq \| \mathcal{I}^N ( f - g ) \|_2 + \| \mathcal{I}^N (g^{+} - f )\|_2 \\
        & \leq 2 \| \mathcal{I}^N (f - g ) \|_2 \leq 2 \left( \| \mathcal{I}^N f - f \|_2 + \| f - \mathcal{I}^N g \|_2 \right).
    \end{align*}
\end{proof}

\section{Proof of \cref{eq:cactboundtoshow}}\label{sec:cacatbound}
Suppose a real orthonormal basis of $\mathbb{M}^{\perp}$ is $\{ \phi_j(x)\}$. We can follow the definitions of $C_{\mathcal{A}}$ and $C_{\mathcal{R}}$ to define matrices $B_{\mathcal{A}}$ and $B_{\mathcal{R}}$, where the elements of $B_{\mathcal{A}}$ and $B_{\mathcal{R}}$ are $\phi_j(x_k)$ with $k \in \mathcal{A}$ and $k \in \mathcal{R} = \mathcal{N}^d \backslash \mathcal{A}$, respectively, and $k$ denotes the row index and $j$ denotes the column index. Thus, by \cref{lemma:ortho}, one can conclude from the orthogonality of the basis functions $\{\psi_j(x)\}$  and $\{\phi_j(x)\}$ that the square matrix
\begin{displaymath}
\sqrt{\frac{(2\pi)^d}{(2N+1)^d}}\begin{pmatrix} C_{\mathcal{A}} & B_{\mathcal{A}} \\ C_{\mathcal{R}} & B_{\mathcal{R}} \end{pmatrix}
\end{displaymath}
is an orthogonal matrix. Therefore,
\begin{equation}\label{eq:cacat}
    C_{\mathcal{A}} C_{\mathcal{A}}^{\intercal} + B_{\mathcal{A}} B_{\mathcal{A}}^{\intercal} = \frac{(2N+1)^d}{(2 \pi)^d} I_{\mathcal{A}}.
\end{equation}
In this section, we are going to show that there exists a constant $C_B$ independent of $N$ such that
\begin{equation}\label{eq:babatbound}
    \left\| \frac{(2 \pi)^d}{(2N+1)^d} B_{\mathcal{A}} B_{\mathcal{A}}^{\intercal} \right\| \leq C_B < 1
\end{equation}
for sufficiently large $N$. Then by \cref{eq:cacat}, we have
\begin{align*}
    \frac{(2N+1)^d}{(2 \pi)^d} \| (C_{\mathcal{A}} C_{\mathcal{A}}^{\intercal} )^{-1}\|_2 & \leq \left\| \left( I_{\mathcal{A}} - \frac{(2 \pi)^d}{(2N+1)^d} B_{\mathcal{A}} B_{\mathcal{A}}^{\intercal} \right)^{-1}  \right\|_2 \\
    & \leq \frac{1}{1 - \left\| \frac{(2 \pi)^d}{(2N+1)^d} B_{\mathcal{A}} B_{\mathcal{A}}^{\intercal} \right\|_2} \leq \frac{1}{1-C_B},
\end{align*}
which proves \cref{eq:cactboundtoshow}.

For simplicity, we define
\begin{displaymath}
\square_k = \left[\frac{(2k_1-1)\pi}{2N+1}, \frac{(2k_1+1)\pi}{2N+1} \right] \times\cdots \times\left[\frac{(2k_d-1)\pi}{2N+1}, \frac{(2k_d+1)\pi}{2N+1} \right], \quad k \in \mathcal{N}^d.
\end{displaymath}
In fact, $\square_k$ denotes a $d$-dimensional hypercube whose center is $x_k$ and the length of edges are $\frac{2\pi}{2N+1}$, and it satisfies $\cup_{k \in \mathcal{N}^d} \square_k = [-\pi, \pi]^d$. We can then describe the main idea of our proof of \cref{eq:babatbound} in the following steps:

\begin{itemize}
    \item For any vector $\boldsymbol{v}$, we can construct a polynomial $P(x)$ such that
    \begin{equation}\label{eq:sumtoint}
        \left| \frac{(2 \pi)^d}{(2N+1)^d} \boldsymbol{v}^{\intercal} B_{\mathcal{A}}^{\intercal} B_{\mathcal{A}} \boldsymbol{v}  -  \sum_{k \in \mathcal{A}} \int_{\square_k} (P(x))^2 \mathrm{d}x \right| \leq C_1 \frac{\boldsymbol{v}^{\intercal}\boldsymbol{v}}{\sqrt{N}}.
    \end{equation}
    Therefore, by triangle inequality, it holds that
    \begin{displaymath}
        \left| \frac{(2 \pi)^d}{(2N+1)^d} \boldsymbol{v}^{\intercal} B_{\mathcal{A}}^{\intercal} B_{\mathcal{A}} \boldsymbol{v} \right| \leq C_1\frac{\boldsymbol{v}^{\intercal}\boldsymbol{v}}{\sqrt{N}} + \sum_{k \in \mathcal{A}} \int_{\square_k} (P(x))^2 \mathrm{d}x.
    \end{displaymath}
    \item The sum of the integrals can be bounded by the integral on $[-\pi,\pi]^d$ as
    \begin{equation}\label{eq:polypartglobal}
        \sum_{k \in \mathcal{A}} \int_{\square_k} (P(x))^2 \mathrm{d}x \leq \mathcal{F}\left( \frac{|\mathcal{A}|}{(2N+1)^d} \right) \| P(x) \|_2^2,
    \end{equation}
    where $\mathcal{F}(x) \geq 0 $ is a strictly increasing function on $[0,1]$ with $\mathcal{F}(1) = 1$.
    \item Then it remains to estimate two terms in the right-hand side of \cref{eq:polypartglobal} as
    \begin{gather}
        \label{eq:activesetbound} \frac{|\mathcal{A}|}{(2N+1)^d} \leq C_2 < 1, \\
        \label{eq:polynormtov} \| P(x) \|_2^2 \leq (1 + \frac{C_3}{\sqrt{N}})^2 \boldsymbol{v}^{\intercal}\boldsymbol{v}. 
    \end{gather}
\end{itemize}
Combining all the inequalities above leads to 
\begin{displaymath}
    \left| \frac{(2 \pi)^d}{(2N+1)^d} \boldsymbol{v}^{\intercal} B_{\mathcal{A}}^{\intercal} B_{\mathcal{A}} \boldsymbol{v} \right| \leq \left( \mathcal{F}(C_2) + \frac{C_1 + 2 \mathcal{F}(C_2) C_3}{\sqrt{N}} + \frac{\mathcal{F}(C_2)C_3^2}{N} \right) \boldsymbol{v}^{\intercal}\boldsymbol{v},
\end{displaymath}
which proves \cref{eq:babatbound} with constant $C_B = \frac{\mathcal{F}(C_2)+1}{2}$ for $N$ sufficiently large. In this section, we will show the proofs of \cref{eq:sumtoint,eq:polypartglobal,eq:activesetbound} in \cref{sec:sumtoint,sec:polypartglobal,sec:activesetbound}, respectively. \cref{eq:polynormtov} is a corollary in the proof of \cref{eq:sumtoint}, which is shown in \cref{sec:sumtoint}.

\subsection{On the Fourier series of polynomials}\label{sec:fourieronpoly}
As the Fourier spectral method is adopted and the introduced $\boldsymbol{m}(x)$ is a vector of polynomials, the remaining proof in this section depends highly on the properties of $\mathcal{P}^N \boldsymbol{m}(x)$. These properties are stated in \cref{lemma:poly}, where its proof is deferred in \cref{appendix:polynomial}.

\begin{lemma}\label{lemma:poly} For a polynomial of degree $K$
\begin{equation*}
    p(x) = \sum_{|\alpha| \leq K} p_\alpha x^{\alpha},\qquad x^{\alpha} = x_1^{\alpha_1} x_2^{\alpha_2} \cdots x_d^{\alpha_d},
\end{equation*}
where $\alpha = (\alpha_1, ..., \alpha_d)^{\intercal} \in \mathbb{N}^d$ is a multi-index and $|\alpha|:=\sum_{j=1}^d \alpha_j$, it holds that
\begin{equation}\label{eq:polygeneral}
    | \mathcal{P}^N p |_{H^1} \leq C(K,d,p) \sqrt{N}, \quad \| \mathcal{P}^N p - p \|_2 \leq C(K,d,p) N^{-1/2},
\end{equation}
where
\begin{displaymath}
C(K,d,p) = \sqrt{2 d \left( \frac{2^d \pi^{K} K^d}{d^d} \right)^2 \left( \frac{\pi^2}{3} + 1 \right)^{d-1}} \sum_{|\alpha| \leq K} |p_{\alpha}|.
\end{displaymath}
\end{lemma}

Since the space of polynomials of degree lower than or equal to $K$ is finite dimensional, we can use the equality of norms in finite dimensional spaces to obtain the following corollary:
\begin{corollary} \label{coro:poly}
    For any polynomial of degree less than or equal to $K$, there exist a constant $\tilde{C}$ depending only on $K$ and $d$ such that
    \begin{displaymath}
    |\mathcal{P}^N p|_{H^1} \leq \tilde{C}\|p\|_2 \sqrt{N}, \qquad \|\mathcal{P}^N p - p\|_2 \leq \tilde{C}\|p\|_2 N^{-1/2}. 
    \end{displaymath}
\end{corollary}

We can now show the linear independence of $\{ \mathcal{P}^N m_j(x),j=1,...,M \}$ under the condition \ref{Hcon1}. 

\begin{lemma}\label{lemma:polycoeffbound}
For a vector of polynomials $\boldsymbol{m}(x)$ satisfying \ref{Hcon1}, there exists $N_0>0$ such that for $N \geq N_0$, it holds that $\|\boldsymbol{a}\| \lesssim \|\boldsymbol{a}^T \mathcal{P}^N \boldsymbol{m}(x)\|_2$ for any real vector $\boldsymbol{a} \in \mathbb{R}^M$.
\end{lemma}

\begin{proof}
For any vector $\boldsymbol{a} = (a_1,...,a_M)^{\intercal}$, we apply \cref{lemma:poly} and triangle inequality to get
\begin{equation*}
\begin{aligned}
\|\boldsymbol{a}^T \mathcal{P}^N \boldsymbol{m}(x)\|_2 & \geq \left\|\sum_{j=1}^{M} a_j m_j(x)\right\|_{2} - \left\|\sum_{j=1}^{M} a_j m_j(x) - \sum_{j=1}^{M} a_j \mathcal{P}^N m_j(x)\right\|_{2} \\
& \geq
\left\|\sum_{j=1}^{M} a_j m_j(x)\right\|_{2} -  \|\boldsymbol{a}\|_2 \sqrt{\sum_{j=1}^M  \left\|m_j(x) - \mathcal{P}^N m_j(x)\right\|_{2}^2} \\
& \geq \left\|\sum_{j=1}^{M} a_j m_j(x)\right\|_{2} - \sigma_1 \|\boldsymbol{a}\|_2 N^{-1/2},
\end{aligned}
\end{equation*}
where the constant $\sigma_1$ depends on the dimension $d$ and the polynomials $\boldsymbol{m}(x)$. On the other hand, the linear independence of $m_j(x)$ (from \ref{Hcon1}) implies that there exists a constant $\sigma_2$ such that $\left\|\sum_{j=1}^{M} a_j m_j(x)\right\|_{2} \geq \sigma_2 \| \boldsymbol{a} \|_2$, where the constant $\sigma_2$ is positive and depends only on the polynomials $\boldsymbol{m}(x)$. Therefore,
\begin{equation}\label{eq:polycoeffbound}
    \|\boldsymbol{a}^T \mathcal{P}^N \boldsymbol{m}(x)\|_2 \geq \left(\sigma_2 - \sigma_1 N^{-1/2} \right) \|\boldsymbol{a}\|_2.
\end{equation}
We can pick a sufficiently large $N_0$ such that the constant in front of $\|\boldsymbol{a}\|_2$ is larger than $\sigma_2/2$ when $N \geq N_0$, which proves the conclusion of the lemma.
\end{proof}

\cref{lemma:polycoeffbound} shows that $\boldsymbol{a}^T \mathcal{P}^N \boldsymbol{m}(x) = 0$ implies  $\boldsymbol{a}=\boldsymbol{0}$, which is summarized in the following corollary. 
\begin{corollary}\label{col:Hcon1toHcon2}
For a vector of polynomials $\boldsymbol{m}(x)$, under the condition \ref{Hcon1}, there exists $N_0>0$ such that for $N \geq N_0$, $\{ \mathcal{P}^N m_j(x) \}$ are linearly independent.
\end{corollary}

\subsection{Proofs of \cref{eq:sumtoint} and \cref{eq:polynormtov}}\label{sec:sumtoint}
We recall that the entries of $B_{\mathcal{A}}$ are $\phi_j(x_k)$ with $k \in \mathcal{A}$ and $j$ being the index of the basis functions of $\mathbb{M}^{\perp}$. On the other hand, from the definition of $\mathbb{M}$ in \cref{eq:mspace}, $\mathbb{M}^{\perp}=\operatorname{span}\{ \mathcal{P}^N m_j(x) \mid j=1,\cdots,M\}$. \cref{col:Hcon1toHcon2} shows that $\mathcal{P}^N m_j(x)$ are linearly independent for sufficiently large $N$, meaning that that $\{\mathcal{P}^N m_j(x)\}$ is also a basis of $\mathbb{M}^{\perp}$. Thus, there exists a matrix $A \in \mathbb{R}^{M\times M}$ such that for any $\boldsymbol{v} = (v_1,\cdots,v_M)^{\intercal}\in \mathbb{R}^M$, 
\begin{equation}\label{eq:vtam}
\sum_{j=1}^M v_j \phi_j(x) = \boldsymbol{v}^{\intercal}A \,\mathcal{P}^N \boldsymbol{m}(x) = \mathcal{P}^N(\boldsymbol{v}^{\intercal} A\boldsymbol{m})(x).
\end{equation}
\cref{lemma:polycoeffbound} and the orthogonality of $\{\phi_j\}$ show that
\begin{equation} \label{eq:vnorm}
\|\boldsymbol{v}\|_2 = \|\boldsymbol{v}^{\intercal} A \,\mathcal{P}^N \boldsymbol{m}(x)\|_2 \gtrsim \|A^{\intercal} \boldsymbol{v}\|_2,
\end{equation}
which indicates that $\|A^{\intercal}\|_2$ is bounded uniformly in $N$.
Below we will choose
\begin{equation} \label{eq:defP}
P(x) = \boldsymbol{v}^{\intercal} A \boldsymbol{m}(x)
\end{equation}
and prove the inequality \cref{eq:sumtoint}. Since
\begin{displaymath}
    \boldsymbol{v}^{\intercal} B_{\mathcal{A}}^{\intercal} B_{\mathcal{A}} \boldsymbol{v} = \sum_{k \in \mathcal{A}} \left(\sum_{j=1}^M v_j \phi_j(x_k) \right)^2 = \sum_{k \in \mathcal{A}} \left( \mathcal{P}^N P(x_k) \right)^2,
\end{displaymath}
we can rewrite the left-hand side of \cref{eq:sumtoint} as
\begin{equation}
    \left| \sum_{k \in \mathcal{A}} \int_{\square_k} \left(\mathcal{P}^N P(x_k) \right)^2 \mathrm{d}x -  \sum_{k \in \mathcal{A}} \int_{\square_k} (P(x))^2 \mathrm{d}x \right|.
\end{equation}
Based on the above reformulation, \cref{eq:sumtoint} can be proved by triangle inequality if
\begin{gather}
    \label{eq:interpP} \left| \sum_{k \in \mathcal{A}} \int_{\square_k} \left(\mathcal{P}^N P(x_k) \right)^2 \mathrm{d}x -  \sum_{k \in \mathcal{A}} \int_{\square_k} (\mathcal{P}^N P(x))^2 \mathrm{d}x \right| \lesssim \frac{\boldsymbol{v}^{\intercal}\boldsymbol{v}}{\sqrt{N}}, \\
    \label{eq:projP} \left| \sum_{k \in \mathcal{A}} \int_{\square_k} \left(\mathcal{P}^N P(x) \right)^2 \mathrm{d}x -  \sum_{k \in \mathcal{A}} \int_{\square_k} (P(x))^2 \mathrm{d}x \right| \lesssim \frac{\boldsymbol{v}^{\intercal}\boldsymbol{v}}{\sqrt{N}}.
\end{gather}
These two inequalities will be shown in this subsection.

\cref{eq:interpP} can be regarded as an interpolation on a portion of collocation points for function $(\mathcal{P}^N P(x) )^2 \in \mathbb{S}^{2N}$. To estimate it, the following lemma is introduced.

\begin{lemma}[Interpolation of trigonometric polynomials]\label{lemma:interpolation}
    For any $\Psi \in \mathbb{S}^{2N}$, it holds that
\begin{equation*}
    \sum_{k \in \mathcal{N}^d} \left| \int_{\square_k} \left( \Psi(x_k) - \Psi(x) \right) \mathrm{d}x \right| \lesssim N^{-1} |\Psi|_{W^{1,1}([-\pi,\pi]^d)}.
\end{equation*}
\end{lemma}

\begin{proof}
    On the one hand, the local quadrature error for smooth function $\Psi$ satisfies (e.g., see \cite[Chapter 3]{raviart1985analysis})
\begin{equation}\label{eq:localintpol}
    \left| \int_{\square_k} \left( \Psi(x_k) - \Psi(x) \right) \mathrm{d}x \right| \lesssim \sum_{j=2}^{\max(3,d+1)}  
N ^{-j} | \Psi |_{W^{j,1}(\square_k)}.
\end{equation}
On the other hand, by the inverse inequality or Bernstein's inequality (e.g., see \cite[Chapter 1.4]{muscalu2013Bernstein}), it holds that for any $1<j \leq d+1$,
\begin{equation*}
    | \Psi |_{W^{j,1}([-\pi,\pi]^d)} \lesssim N^{j-1} | \Psi |_{W^{1,1}([-\pi,\pi]^d)}.
\end{equation*}
As a result, we take a sum of $k \in \mathcal{N}^d$ in \cref{eq:localintpol} and apply the above inverse inequality to obtain
\begin{displaymath}
\begin{aligned}
   & \sum_{k \in \mathcal{N}^d} \left| \int_{\square_k} \left( \Psi(x_k) - \Psi(x) \right) \mathrm{d}x \right| \\
   \lesssim & \sum_{j=2}^{\max(3,d+1)}  
N ^{-j} | \Psi |_{W^{j,1}([-\pi,\pi]^d)} \lesssim N^{-1} | \Psi |_{W^{1,1}([-\pi,\pi]^d)}.
\end{aligned}
\end{displaymath}
\end{proof}

We are now ready to prove the inequalities \cref{eq:interpP,eq:projP}. During the proof, the inequality \cref{eq:polynormtov} will be used as an intermediate result. Below we will present three proofs for these three inequalities.
\begin{proof}[Proof of \cref{eq:interpP}]
\cref{lemma:interpolation} shows
\begin{equation}\label{eq:localhighderi}
\begin{aligned}
    & \left| \sum_{k \in \mathcal{A}} \int_{\square_k} \left(\mathcal{P}^N P(x_k) \right)^2 \mathrm{d}x -  \sum_{k \in \mathcal{A}} \int_{\square_k} (\mathcal{P}^N P(x))^2 \mathrm{d}x \right| \\
    \leq {} & \sum_{k \in \mathcal{N}^d} \left| \int_{\square_k}  \left( (\mathcal{P}^N P(x_k))^2 - (\mathcal{P}^N P(x))^2 \right) \mathrm{d}x \right| \\
    \lesssim {} & N^{-1} | (\mathcal{P}^N P(x))^2 |_{W^{1,1}([-\pi,\pi]^d)},
\end{aligned}
\end{equation}
By the chain rule,
\begin{equation}\label{eq:pnpw11norm}
\begin{aligned}
    | (\mathcal{P}^N P(x))^2 |_{W^{1,1}([-\pi,\pi]^d)} & = 2 \int_{[-\pi,\pi]^d} |\nabla \mathcal{P}^N P(x)|_1 |\mathcal{P}^N P(x)|\mathrm{d}x \\
    & \leq 2 \sqrt{d} \| \mathcal{P}^N P(x)\|_2 | \mathcal{P}^N P(x) |_{H^1} = 2 \sqrt{d} \| \boldsymbol{v} \|_2 | \mathcal{P}^N P(x) |_{H^1},
\end{aligned}
\end{equation}
where the last equality utilizes the equality in \cref{eq:vnorm}. Then according to \cref{eq:vtam,eq:vnorm,eq:defP} and \cref{lemma:poly},
\begin{equation}\label{eq:pnph1norm}
    | \mathcal{P}^N P(x) |_{H^1} \leq \|\boldsymbol{A^{\intercal} v}\|_2\sqrt{\sum_{j=1}^M | \mathcal{P}^N m_j(x) |_{H^1}^2}\lesssim
    \|\boldsymbol{v}\|_2 \sqrt{N}.
\end{equation}
Plugging \cref{eq:pnpw11norm,eq:pnph1norm} into \cref{eq:localhighderi} proves \cref{eq:interpP}.
\end{proof}
\begin{proof}[Proof of \cref{eq:polynormtov}]
By \cref{coro:poly} and \cref{eq:vtam,eq:vnorm,eq:defP}, we get the following estimate of the projection error for sufficiently large $N$: 
\begin{equation}\label{eq:colbound}
\|  \mathcal{P}^N P - P \|_2  \lesssim N^{-1/2} \|P\|_2 \leq N^{-1/2} \|A^{\intercal} \boldsymbol{v}\|_2  \|\boldsymbol{m}^{\intercal} \boldsymbol{m}\|_{1}^{1/2} \lesssim N^{-1/2} \|\boldsymbol{v}\|_2.
\end{equation}
Assume that the constant in this equality is $C_3$, i.e., $\|  \mathcal{P}^N P - P \|_2 \leq C_3 N^{-1/2} \|\boldsymbol{v}\|_2$. Then by triangle inequality,
\begin{equation}\label{eq:Pboundpn}
\| P \|_2 \leq \| \mathcal{P}^N P\|_2 + \|  \mathcal{P}^N P - P \|_2
\leq \left( 1 + \frac{C_3}{\sqrt{N}} \right) \| \boldsymbol{v} \|_2.
\end{equation}
Taking the square on both sides of \cref{eq:Pboundpn} yields \cref{eq:polynormtov}.
\end{proof}

\begin{proof}[Proof of \cref{eq:projP}]
The left-hand side of \cref{eq:projP} satisfies
\begin{equation}\label{eq:projerr1}
    \begin{aligned}
     & \left| \sum_{k \in \mathcal{A}} \int_{\square_k} \left( (\mathcal{P}^N P(x))^2 - (P(x))^2 \right) \mathrm{d}x \right| \\
    \leq & \| \left(\mathcal{P}^N P + P \right) \left( \mathcal{P}^N P - P \right) \|_1 \\ 
     \leq & \| \mathcal{P}^N P + P \|_2 \|  \mathcal{P}^N P - P \|_2 \leq 2 \| P \|_2 \|  \mathcal{P}^N P - P \|_2,
    \end{aligned}
\end{equation}
where $\| \mathcal{P}^N P \|_2 \leq \| P \|_2$ is used in the last inequality. We can now apply \cref{eq:colbound} and \cref{eq:Pboundpn} to obtain
\begin{displaymath}
     \left| \sum_{k \in \mathcal{A}} \int_{\square_k} \left( (\mathcal{P}^N P(x))^2 - (P(x))^2 \right) \mathrm{d}x \right| \lesssim N^{-1/2} \left( 1 + \frac{C_3}{\sqrt{N}} \right)\|\boldsymbol{v}\|_2^2 \lesssim N^{-1/2} \|\boldsymbol{v}\|_2^2.
\end{displaymath}
\end{proof}

\subsection{Proof of \cref{eq:polypartglobal}}\label{sec:polypartglobal}
The proof of \cref{eq:polypartglobal} requires a theorem introduced in \cite{brudnyi1973extremal}. Here we rewrite it as a lemma:
\begin{lemma}[Remark 3 in \cite{brudnyi1973extremal}]\label{lemma:polyl2}
Let $P(x)$ be an arbitrary polynomial of degree $K$ and $V$ be a $d$-dimensional convex set of positive measure. Then for any subset $\Omega \subset V$ of positive measure,
\begin{equation}
    \left( \frac{1}{|V|} \int_V |P(x)|^2 \mathrm{d}x \right)^{1/2} \leq C(d,K) \left(\frac{|\Omega|}{|V|} \right)^{-K} \left( \frac{1}{|\Omega|} \int_{\Omega} |P(x)|^2 \mathrm{d}x \right)^{1/2},
\end{equation}
where constant $C(d,K) \geq 1$ depends on $d$ and $K$.
\end{lemma}

Based on the notations in \cref{lemma:polyl2}, we plug $V = [-\pi,\pi]^d$, $\bar{\Omega} = \cup_{k \in \mathcal{A}} \square_k$, $\Omega = V \backslash \bar{\Omega}$ and $P(x)$ from \cref{eq:defP} with degree $K$ into \cref{lemma:polyl2},
\begin{equation*}
    \int_V |P(x)|^2 \mathrm{d}x \leq (C(d,K))^2 \left(\frac{|\Omega|}{|V|} \right)^{-2K-1}  \int_{\Omega} |P(x)|^2 \mathrm{d}x.
\end{equation*}
Since $\Omega = V \backslash \bar{\Omega}$, the above inequality can be written as
\begin{equation}\label{eq:polylocglobal}
    \int_{\bar{\Omega}} |P(x)|^2 \mathrm{d}x \leq \left( 1 - (C(d,K))^{-2} \left(1 - \frac{|\bar{\Omega}|}{|V|} \right)^{2K+1} \right) \int_{V} |P(x)|^2 \mathrm{d}x.
\end{equation}
It is easy to see that $\frac{|\bar{\Omega}|}{|V|} = \frac{|\mathcal{A}|}{(2N+1)^d}$. Therefore, \cref{eq:polylocglobal} shows that \cref{eq:polypartglobal} holds with function $\mathcal{F}(x)$
\begin{displaymath}
    \mathcal{F}(x) = 1 - (C(d,K))^{-2} \left(1 - x \right)^{2K+1}.
\end{displaymath}

\subsection{Proof of \cref{eq:activesetbound}}\label{sec:activesetbound}
We will estimate the size of active set to show $\frac{|\mathcal{A}|}{(2N+1)^d} \leq C_2<1$, which is based on the construction of a series of positive and moment-preserving functions. This series of functions will converge to a $L^2$ function. The construction and related proof are mainly from \cite{mieussens2000discrete,mieussens2001convergence}.

\begin{lemma}\label{lemma:twoopt}
    Consider two optimization problems
    \begin{align}
        \label{eq:optJN} & J_{N}(\boldsymbol{\alpha}_N) = \min_{\boldsymbol{\beta} \in \mathbb{R}^M} \{ J_N(\boldsymbol{\beta}) =  \frac{(2\pi)^d}{(2N+1)^d} \sum_{k \in \mathcal{N}^d} \exp(\boldsymbol{\beta} \cdot \mathcal{P}^N \boldsymbol{m}(x_k) )- \boldsymbol{\beta} \cdot \boldsymbol{\rho}(f) \},\\
        \label{eq:optJ} & J(\boldsymbol{\alpha}) = \min_{\boldsymbol{\beta} \in \mathbb{R}^M} \{ J(\boldsymbol{\beta}) = \langle \exp(\boldsymbol{\beta} \cdot \boldsymbol{m}(x))\rangle - \boldsymbol{\beta} \cdot \boldsymbol{\rho}(f)\}.
    \end{align}
    Under \ref{Hcon1}--\ref{Hcon4} and for sufficiently large $N$, both $J_N$ and $J$ have unique solutions $\boldsymbol{\alpha}_N$ and $\boldsymbol{\alpha}$, respectively. Moreover, $J_N$ is locally uniformly convergent to $J$ on $\mathbb{R}^M$, and $\boldsymbol{\alpha}_N$ converges to $\boldsymbol{\alpha}$.
\end{lemma}

The proof of this lemma is almost identical to the proof of Theorem 3.1 in \cite{mieussens2000discrete} (see also the proof of Theorem 1 in \cite{mieussens2001convergence}). We omit the details here.

\begin{lemma}[Proof of \cref{eq:activesetbound}]\label{lemma:activeset}
    For the active set $\mathcal{A} = \mathcal{A}(N)$ defined in \cref{eq:optactive}, under the conditions of \cref{lemma:twoopt}, there exist constants $0 \leq C_2<1$ and $N_0>0$ such that for all $N \geq N_0$, \cref{eq:activesetbound} is satisfied.
\end{lemma}

\begin{proof}
We will prove this lemma by contradiction. If the lemma is not true, for any $\varepsilon>0$ and $N_0$, there exists $N>N_0$ such that
\begin{equation*}
    \frac{|\mathcal{A}|}{(2N+1)^d} > 1 - \varepsilon.
\end{equation*}
Recalling $\mathcal{R}$ is the complement set of $\mathcal{A}$ such that $\mathcal{R} = \mathcal{N}^d \backslash \mathcal{A}$, it holds that
\begin{equation*}
    \frac{|\mathcal{R}|}{(2N+1)^d} \leq \varepsilon.
\end{equation*}
By definition, for all $k \in \mathcal{A}$, we have $\Pi^N_+(f_N^c)(x_k)=0$. Therefore, according to the Cauchy–Schwarz inequality and \cref{eq:sninnerprod}, it holds that
\begin{equation*}
    \begin{split}
        \langle \Pi^N_+(f_N^c) \rangle & = \frac{(2 \pi)^d}{(2N+1)^d} \sum_{k \in \mathcal{N}^d} \Pi^N_+(f_N^c)(x_k)
        = \frac{(2 \pi)^d}{(2N+1)^d} \sum_{k \in \mathcal{R}} \Pi^N_+(f_N^c)(x_k) \\
        & \leq \sqrt{\frac{(2\pi)^d|\mathcal{R}|}{(2N+1)^d}} \left( \frac{(2 \pi)^d}{(2N+1)^d} \sum_{k \in \mathcal{R}} \left( \Pi^N_+(f_N^c)(x_k) \right)^2 \right)^{1/2} \\
        & = \sqrt{\frac{(2\pi)^d|\mathcal{R}|}{(2N+1)^d}} \| \Pi^N_+(f_N^c) \|_2 \leq \sqrt{(2 \pi)^d \varepsilon} \| \Pi^N_+(f_N^c) \|_2
    \end{split}
\end{equation*}
On the other hand, the left-hand side $\langle \Pi^N_+(f_N^c) \rangle = \langle f \rangle >0$. Therefore,
\begin{equation}\label{eq:pintoinf}
    \| \Pi^N_+(f_N^c) \|_2 \geq \frac{\langle f \rangle}{\sqrt{(2 \pi)^d \varepsilon}}.
\end{equation}
In particular, we take $\varepsilon = 1 / N_0$ and let $N_0$ go to infinity to construct a subsequence of $\| \Pi^N_+(f_N^c) \|_2$ that goes to infinity. Together with the uniform boundedness of $\| f_N^c - f \|_2$ \cite{rey2022sinum}, it holds that the subsequence of $\| \Pi^N_+(f_N^c) - f_N^c \|_2$ also goes to infinity, which means there exists a subsequence of the optimal objective functions in \cref{eq:ftom} that tends to infinity.

Now we start to show that there is a feasible solution of \cref{eq:ftom} for each sufficiently large $N$, and the $L^2$ norms of these feasible solutions are uniformly bounded. This means the optimal value of \cref{eq:ftom} is uniformly bounded.

According to \cref{lemma:twoopt}, the optimization problem \cref{eq:optJN} has a unique solution $\boldsymbol{\alpha}_N$ for any sufficiently large $N$. The first-order optimality condition of \cref{eq:optJN} shows
\begin{equation}\label{eq:dismoment}
    \frac{(2\pi)^d}{(2N+1)^d} \sum_{k \in \mathcal{N}^d} \exp(\boldsymbol{\alpha}_N \cdot \mathcal{P}^N \boldsymbol{m}(x_k) ) \mathcal{P}^N \boldsymbol{m}(x_k) = \boldsymbol{\rho}(f).
\end{equation}
Then, the function 
\begin{displaymath}
\mathcal{E}_N(x) = \mathcal{I}^N (\exp(\boldsymbol{\alpha}_N \cdot \mathcal{P}^N \boldsymbol{m}(x) )) \in \mathbb{S}^N
\end{displaymath}
satisfies $\mathcal{E}_N(x_k) > 0$ for all $k \in \mathcal{N}^d$ and $\boldsymbol{\rho}(\mathcal{E}_N) = \boldsymbol{\rho}(f)$. Thus, using triangle inequality and the optimality of $\Pi_N^+(f_N^c)$, we have
\begin{displaymath}
\|\mathcal{E}_N\|_2 \geq \|\mathcal{E}_N - f_N^c\|_2 - \|f_N^c\|_2 \geq \|\Pi_N^+(f_N^c) - f_N^c\|_2 - \|f_N^c\|_2,
\end{displaymath}
which implies that $\|\mathcal{E}_N\|_2$ also diverges to infinity in the subsequence.

However, by \cref{lemma:twoopt}, $\boldsymbol{\alpha}_N$ converges to $\boldsymbol{\alpha}$, which is the unique solution of \cref{eq:optJ}. Making use of \cref{eq:sninnerprod} and the (local) uniform convergence of $J_N$ to $J$ in \cref{lemma:twoopt}, this sequence satisfies
\begin{displaymath}
    \| \mathcal{E}_N(x) \|_2^2 = \frac{(2 \pi)^d}{(2N+1)^d} \sum_{k \in \mathcal{N}^d} \exp(2 \boldsymbol{\alpha}_N \cdot \mathcal{P}^N \boldsymbol{m}(x_k) ) \to \langle \exp(2 \boldsymbol{\alpha} \cdot \boldsymbol{m}(x)) \rangle,
\end{displaymath}
which is bounded. This contradicts our previous conclusion that the subsequence of $\|\mathcal{E}_N\|_2$ tends to infinity. By the method of contradiction, the constant $C_2 < 1$ must exist.
\end{proof}

\section{An efficient solver of the optimization problem}\label{sec:algorithm}
In this section, we aim at designing an efficient numerical scheme to solve \cref{eq:optieq}.

To brief the idea more clearly, we can further apply \cref{eq:sninnerprod} to \cref{eq:parseopt} and obtain an equivalent form of the minimization problem in \cref{eq:optieq} as
\begin{equation*}
    \min_{g_k} \sum_{k \in \mathcal{N}^d} \left( g_k - (\mathcal{P}^N f)(x_k) \right)^2 \ \ \text{  s.t. } g_k \geq 0 \text{ and } \frac{(2 \pi)^d}{(2N+1)^d} \sum_{k \in \mathcal{N}^d} \mathcal{P}^N \boldsymbol{m}(x_k) g_k = \boldsymbol{\rho}(f).
\end{equation*}
The above optimization problem can be represented as a matrix-vector form as
\begin{equation}\label{eq:optfinite}
    \min_{\boldsymbol{g}} \| \boldsymbol{g} - \boldsymbol{f}_N \|_2^2, \qquad \text{ s.t. } \boldsymbol{g} \in \mathbb{R}^{(2N+1)^d}_+ \text{ and } \mathcal{M} \boldsymbol{g} = \boldsymbol{\rho}(f),
\end{equation}
where $\boldsymbol{f}_N$ is the vector of $(\mathcal{P}^N f)(x_k)$ for $k \in \mathcal{N}^d$; the entries of the matrix $\mathcal{M} \in \mathbb{R}^{M \times (2N+1)^d}$ is $\frac{(2 \pi)^d}{(2N+1)^d} \mathcal{P}^N m_j(x_k)$, and $j$ denotes the row index and $k$ denotes the column index. By \ref{Hcon4}, we have $\boldsymbol{\rho}(f)\in \operatorname{Ran}(\mathcal{M})$, where ${\rm Ran}(\mathcal{M})$ is the range space of $\mathcal{M}$. According to \cref{sec:wellpose}, \eqref{eq:optfinite} admits a unique minimizer, denoted as $\boldsymbol{g}^*$. Furthermore, \cref{col:Hcon1toHcon2} and \cref{eq:sninnerprod} implies ${\rm Ran}(\mathcal{M}) = \mathbb{R}^M$ for sufficiently large $N$. 

The problem \cref{eq:optfinite} is a convex quadratic programming problem, which computes the projection of $\boldsymbol{f}_N $ onto a polyhedral set $\{\boldsymbol{g}\in \mathbb{R}^{(2N+1)^d}_+: \mathcal{M} \boldsymbol{g} = \boldsymbol{\rho}(f)\}$. By introducing a Lagrange multiplier $\boldsymbol{\lambda}\in {\rm Ran}(\mathcal{M})$, the Lagrange function associated with the problem \cref{eq:optfinite} is given as
\begin{equation}\label{eq:lagfun}
    L(\boldsymbol{g}; \boldsymbol{\lambda}) = \| \boldsymbol{g} - \boldsymbol{f}_N\|_2^2 +\delta_{+}(\boldsymbol{g}) - \boldsymbol{\lambda}^{\intercal} \left( \mathcal{M} \boldsymbol{g} - \boldsymbol{\rho}(f) \right), \ (\boldsymbol{g},\boldsymbol{\lambda})\in \mathbb{R}^{(2N+1)^d}\times {\rm Ran}(\mathcal{M}),
\end{equation}
where $\delta_{+}(\cdot)$ is the indicator function of the non-negative orthant $\mathbb{R}_{+}^{(2N+1)^d}$. Based on \cref{eq:lagfun}, the dual problem of \cref{eq:optfinite} takes the form as 
\begin{align}
\max_{\boldsymbol{\lambda} \in {\rm Ran}(\mathcal{M})} \left\{ \min_{\boldsymbol{g} \in \mathbb{R}^{(2N+1)^d}} L(\boldsymbol{g}; \boldsymbol{\lambda}) \right\}.\label{eq:optdual_org}
\end{align}
By rewriting the Lagrange function $L(\boldsymbol{g};\boldsymbol{\lambda})$ as
\begin{displaymath}
L(\boldsymbol{g}; \boldsymbol{\lambda})
 = \left\| \boldsymbol{g} - \left( \boldsymbol{f}_N + \frac{1}{2} \mathcal{M}^{\intercal} \boldsymbol{\lambda} \right) \right\|_2^2 +\delta_{+}(\boldsymbol{g})+ \boldsymbol{\lambda}^{\intercal} \boldsymbol{\rho}(f) - \left\| \boldsymbol{f}_N + \frac{1}{2} \mathcal{M}^{\intercal} \boldsymbol{\lambda} \right\|_2^2 + \left\| \boldsymbol{f}_N \right\|_2^2,
\end{displaymath}
we can find that $\boldsymbol{g} = \Pi_+ \left( \boldsymbol{f}_N + \frac{1}{2} \mathcal{M}^{\intercal} \boldsymbol{\lambda} \right)$ minimizes $L(\boldsymbol{g};\boldsymbol{\lambda})$ for any fixed $\boldsymbol{\lambda}$, where $\Pi_+(\cdot)$ is the projection operator onto $\mathbb{R}^{(2N+1)^d}_+$ defined as: for any $\boldsymbol{w}\in \mathbb{R}^{(2N+1)^d}$,
\begin{equation*}
    \Big(\Pi_+ (\boldsymbol{w})\Big)_k = \max(0, w_k), \qquad \forall k.
\end{equation*}
Then the dual problem \cref{eq:optdual_org} can be written equivalently as
\begin{equation}\label{eq:optdual}
    \max_{\boldsymbol{\lambda} \in {\rm Ran}(\mathcal{M})} \left\{
    \Phi(\boldsymbol{\lambda}):=- \left\| \Pi_{+} \left( \boldsymbol{f}_N + \frac{1}{2} \mathcal{M}^{\intercal} \boldsymbol{\lambda} \right) \right\|_2^2 + \boldsymbol{\lambda}^{\intercal} \boldsymbol{\rho}(f) + \left\| \boldsymbol{f}_N\right\|_2^2 
    \right\}.
\end{equation}
Note that the subspace constraint $\boldsymbol{\lambda} \in {\rm Ran}(\mathcal{M})$ is imposed to ensure the boundedness of the solution set of the dual problem \cref{eq:optdual}. 

We are going to focus on solving the dual problem \cref{eq:optdual}, which is an unconstrained maximization problem on ${\rm Ran}(\mathcal{M})$ with a concave objective function. As long as we obtain a maximizer $\boldsymbol{\lambda}^*$ of \cref{eq:optdual}, we can compute the unique optimal solution to \cref{eq:optfinite} as $\boldsymbol{g}^* = \Pi_{+} \left( \boldsymbol{f}_N + \frac{1}{2} \mathcal{M}^{\intercal} \boldsymbol{\lambda^*} \right)$.

According to \cite{moreau1965proximite}, we can see that $\Phi(\cdot)$ is continuously differentiable with respect to $\boldsymbol{\lambda}$ on ${\rm Ran}(\mathcal{M})$ with
\begin{align*}
\nabla \Phi(\boldsymbol{\lambda}) = - \mathcal{M} \Pi_{+} \left( \boldsymbol{f}_N + \frac{1}{2} \mathcal{M}^{\intercal} \boldsymbol{\lambda} \right) + \boldsymbol{\rho}(f),\quad \mbox{for any }\boldsymbol{\lambda}\in {\rm Ran}(\mathcal{M}).
\end{align*}
Then the dual problem \cref{eq:optdual} can be solved by the nonsmooth equation
\begin{align}
\nabla \Phi(\boldsymbol{\lambda}) = 0,\quad \boldsymbol{\lambda}\in {\rm Ran}(\mathcal{M}).\label{eq:dualopt}
\end{align}
Noting the fact that $\nabla \Phi(\cdot)$ is nondifferentiable, we cannot apply the commonly-used derivative-based methods directly. Fortunately, since $\Pi_{+}(\cdot)$ is strongly semismooth \cite{sun2008lowner}, we can develop a semismooth Newton method to solve the nonsmooth equation \cref{eq:dualopt} and expect a quadratic convergence rate. 

We present our proposed method for solving \cref{eq:optfinite} in \cref{alg:ssn}, together with its convergence results in \cref{thm: ssn}. The detailed proof can be found in \cref{sec: convergence_SSN}.

\begin{algorithm}[h]
	\caption{A semismooth Newton based algorithm for solving \cref{eq:optfinite}}
	\label{alg:ssn}
	\begin{algorithmic}[1]
		\STATE \textbf{Initialization}. Given $\tau_1,\tau_2\in (0,1)$ to ensure the non-negative definiteness of Hessian matrices and $\mu\in (0,1/2)$, $\delta\in (0,1)$ to set line search step sizes. Let the initial point $\boldsymbol{\lambda}^0$ be the zero vector in $\mathbb{R}^M$. Set $j=0$.
        \REPEAT
		\STATE {\bfseries Step 1 (Newton Direction)} Let the vector $\boldsymbol{u}^j\in \mathbb{R}^{(2N+1)^d}$ be defined as: for each $k$, $(\boldsymbol{u}^j)_k = 1$ if $(\boldsymbol{f}_N + \frac{1}{2} \mathcal{M}^{\intercal} \boldsymbol{\lambda}^j )_k>0$ and $(\boldsymbol{u}^j)_k =0$ otherwise. Set $\varepsilon_j:=\tau_1 \min \{\tau_2,\|\nabla \Phi(\boldsymbol{\lambda}^j)\|\}$ and solve $\boldsymbol{d}^j\in {\rm Ran}(\mathcal{M})$ from the linear system 
	  \begin{align*}
	  \left(- \frac{1}{2} \mathcal{M} \text{Diag}(\boldsymbol{u}^j) \mathcal{M}^{\intercal} - \varepsilon_j I_M \right) \boldsymbol{d} = -\nabla \Phi(\boldsymbol{\lambda}^j).
		\end{align*}
        \\[3pt]
		\STATE {\bfseries Step 2 (Line Search)} Let $n_j$ be the smallest non-negative integer $n$ for which
        \begin{align*}
        \Phi(\boldsymbol{\lambda}^j + \delta^n \boldsymbol{d}^j)\geq \Phi(\boldsymbol{\lambda}^j) + \mu \delta^n \langle \nabla \Phi(\boldsymbol{\lambda}^j), \boldsymbol{d}^j\rangle .
        \end{align*}
	   \\[3pt]
	 \STATE {\bfseries Step 3}. Set 
        $\boldsymbol{\lambda}^{j+1} = \boldsymbol{\lambda}^j +\delta^{n_j} \boldsymbol{d}^j$, and $j\leftarrow j+1$.
        \UNTIL{$\| \nabla \Phi(\boldsymbol{\lambda}^j) \|$ is smaller than the given tolerance.}
    \RETURN An approximate solution $\hat{\boldsymbol{g}} := \Pi_{+} \left( \boldsymbol{f}_N + \frac{1}{2} \mathcal{M}^{\intercal} \boldsymbol{\lambda}^j \right)$ to \cref{eq:optfinite}.
	\end{algorithmic}
\end{algorithm}

\begin{theorem}\label{thm: ssn}
Let $\{\boldsymbol{\lambda}^j\}$ be the infinite sequence generated by Algorithm \ref{alg:ssn}. Then $\{\boldsymbol{\lambda}^j\}\subseteq {\rm Ran}(\mathcal{M})$ is a bounded sequence and any accumulation point is an optimal solution to the problem \cref{eq:optdual}.

Furthermore, let ${\cal K}:=\{k \in \mathcal{N}^d : (\boldsymbol{g}^*)_k >0\}$ and assume
\begin{align}
{\rm span} (\{ \mathcal{M}_{:,k} \}_{k\in {\cal K}} )={\rm Ran}(\mathcal{M}),\label{eq:span_condition}
\end{align}
where $\mathcal{M}_{:,k}$ is the $k$th column of the matrix $\mathcal{M}$, then the sequence $\{\boldsymbol{\lambda}^j\}$ converges to an optimal solution $\boldsymbol{\lambda}^*$ to \cref{eq:optdual} and
\begin{align*}
\|\boldsymbol{\lambda}^{j+1}-\boldsymbol{\lambda}^*\|\leq \mathcal{O}(\|\boldsymbol{\lambda}^j-\boldsymbol{\lambda}^*\|^{2}).
\end{align*}
\end{theorem}

\section{Numerical results}\label{sec:numerical}
In this section, we will present numerical results to demonstrate the performance of our positive and moment-preserving Fourier spectral method. We will take the number of dimensions $d=3$ in all the numerical examples and denote the spatial variable as $(x,y,z)^{\intercal} \in \mathbb{R}^3$.

\subsection{Convergence test}
In this example, we choose different functions $f$ and check the error of $\Pi_+^N f$. Since $\Pi_+^N f \in \mathbb{S}^N$, we can apply Parseval's theorem and obtain
\begin{displaymath}
    \| f - \Pi_+^N f \|_2^2 = \| f - \mathcal{P}^N f \|_2^2 + \| \mathcal{P}^N f - \Pi_+^N f \|_2^2.
\end{displaymath}
All the three errors in the above equality will be presented for a detailed comparison.

\paragraph{Example 1: $H^m_p$ functions} In the first example, we consider periodic functions defined on $[-\pi,\pi]^3$ with the following expression:
\begin{equation}\label{eq:cosinefun}
    f(x,y,z) =\left[ \left( 1 - \cos(x) \right) \left( 1 - \cos(y) \right) \left( 1 - \cos(z) \right) \right]^{q}.
\end{equation}
Three choices of $q$ will be studied:
\begin{itemize}
\item $q=4/5$, for which $f \in H^2_p([-\pi,\pi]^3)$ but $f \not\in H^3_p([-\pi,\pi]^3)$.
\item $q=13/10$, for which $f \in H^3_p([-\pi,\pi]^3)$ but $f \not\in H^4_p([-\pi,\pi]^3)$.
\item $q=9/5$, for which $f \in H^4_p([-\pi,\pi]^3)$ but $f \not\in H^5_p([-\pi,\pi]^3)$.
\end{itemize}

\begin{table}[!ht]
    \renewcommand{\arraystretch}{1.} 
    \centering
    \caption{Errors of projections for $f = \left[\left( 1 - \cos(x) \right) \left( 1 - \cos(y) \right) \left( 1 - \cos(z) \right) \right]^{q}$.}
    \resizebox{\textwidth}{!}{
    \begin{tabular}{cccccc}
        \hline
        $N$ & $2$ & $4$ & $8$ & $16$ & $32$ \\ 
        \hline
        \multicolumn{6}{c}{$q=4/5$}\\
        \hline
        $\| \mathcal{P}^N f - \Pi_+^N f \|_2$ &  $2.94 \times 10^{-2}$ &  $2.28 \times 10^{-3}$ &  $1.31 \times 10^{-4}$ &  $6.35 \times 10^{-6}$ & $2.83 \times 10^{-7}$ \\
        $\| \mathcal{P}^N f - \Pi^N f \|_2$ & $2.85 \times 10^{-2}$ &  $2.26 \times 10^{-3}$ &  $1.31 \times 10^{-4}$ &  $6.35 \times 10^{-6}$ & $2.83 \times 10^{-7}$ \\
        $\| f - \mathcal{P}^N f \|_2$ &  $5.27 \times 10^{-1}$ &  $1.54 \times 10^{-1}$ &  $4.04 \times 10^{-2}$ &  $1.00 \times 10^{-2}$ & $2.42 \times 10^{-3}$ \\
        \hline
        $\| f - \Pi_+^N f \|_2$ &  $5.28 \times 10^{-1}$ &  $1.54 \times 10^{-1}$ &  $4.04 \times 10^{-2}$ &  $1.00 \times 10^{-2}$ & $2.42 \times 10^{-3}$ \\
        order &   &  $1.78$ &  $1.93$ &  $2.01$ & $2.05$ \\
        \hline
        \multicolumn{6}{c}{$q=13/10$}\\
        \hline
        $\| \mathcal{P}^N f - \Pi_+^N f \|_2$ &  $9.08 \times 10^{-1}$ &  $1.23 \times 10^{-1}$ &  $1.61 \times 10^{-2}$ &  $2.00 \times 10^{-3}$ & $2.42 \times 10^{-4}$ \\
        $\| \mathcal{P}^N f - \Pi^N f \|_2$ & $4.06 \times 10^{-2}$ &  $1.74 \times 10^{-3}$ &  $5.34 \times 10^{-5}$ &  $1.33 \times 10^{-6}$ & $3.02 \times 10^{-8}$ \\
        $\| f - \mathcal{P}^N f \|_2$ &  $7.69 \times 10^{-1}$ &  $1.18 \times 10^{-1}$ &  $1.61 \times 10^{-2}$ &  $2.05 \times 10^{-3}$ & $2.51 \times 10^{-4}$ \\
        \hline
        $\| f - \Pi_+^N f  \|_2$ &  $1.19 \times 10^{0}$ &  $1.70 \times 10^{-1}$ &  $2.28 \times 10^{-2}$ &  $2.86 \times 10^{-3}$ & $3.49 \times 10^{-4}$ \\
        order &   &  $2.81$ &  $2.90$ &  $2.99$ & $3.03$ \\
        \hline
        \multicolumn{6}{c}{$q=9/5$}\\
        \hline
        $\| \mathcal{P}^N f - \Pi_+^N f \|_2$ &  $1.02 \times 10^{0}$ &  $6.61 \times 10^{-2}$ &  $4.40 \times 10^{-3}$ &  $2.80 \times 10^{-4}$ & $1.71 \times 10^{-5}$ \\
        $\| \mathcal{P}^N f - \Pi^N f \|_2$ & $5.41 \times 10^{-2}$ &  $1.13 \times 10^{-3}$ &  $1.78 \times 10^{-5}$ &  $2.28 \times 10^{-7}$ & $2.59 \times 10^{-9}$ \\
        $\| f - \mathcal{P}^N f \|_2$ &  $1.06 \times 10^{0}$ &  $7.67 \times 10^{-2}$ &  $5.33 \times 10^{-3}$ &  $3.46 \times 10^{-4}$ & $2.14 \times 10^{-5}$ \\
        \hline
        $\| f - \Pi_+^N f \|_2$ &  $1.47 \times 10^{0}$ &  $1.01 \times 10^{-1}$ &  $6.91 \times 10^{-3}$ &  $4.45 \times 10^{-4}$ & $2.74 \times 10^{-5}$ \\
        order &   &  $3.86$ &  $3.87$ &  $3.96$ & $4.02$ \\
        \hline
    \end{tabular}
    }
    \label{tab:projerror3}
\end{table}

\cref{tab:projerror3} shows that the error of our new projection $\| f - \Pi_+^N f \|_2$ is always at the same magnitude as the error of the direct projection $\| f - \mathcal{P}^N f \|_2$. Moreover, the convergence order of $\| f - \Pi_+^N f \|_2$ differs with respect to the regularity of function $f$, which agrees with \cref{cor:spectral}.

\paragraph{Example 2: smooth function} The second example is a smooth and periodic function
\begin{equation}\label{eq:funsinexp}
    f(x,y,z) = \left(\sin^2x + \sin^2 y + \sin^2 z \right) \exp \left(- \frac{1}{\sin^2 x} - \frac{1}{\sin^2 y} - \frac{1}{\sin^2 z}  \right),
\end{equation}
for which the spectral accuracy of $\| f - \Pi_+^N f \|_2$ can be observed in \cref{tab:projerror6}.

\begin{table}[h]
    \renewcommand{\arraystretch}{1.} 
    \centering
    \caption{Errors of the different projections for $f$ in \cref{eq:funsinexp}.}
    \resizebox{\textwidth}{!}{
    \begin{tabular}{cccccc}
        \hline
        $N$ & $2$ & $4$ & $8$ & $16$ & $32$ \\ 
        \hline
        $\| \mathcal{P}^N f - \Pi_+^N f \|_2$ &  $1.08 \times 10^{-1}$ &  $1.40 \times 10^{-2}$ &  $3.80 \times 10^{-3}$ &  $2.14 \times 10^{-4}$ & $6.03 \times 10^{-6}$ \\
        $\| \mathcal{P}^N f - \Pi^N f \|_2$ & $3.35 \times 10^{-3}$ &  $4.22 \times 10^{-4}$ &  $3.88 \times 10^{-5}$ &  $1.02 \times 10^{-6}$ & $7.13 \times 10^{-9}$ \\
        $\| f - \mathcal{P}^N f \|_2$ &  $6.24 \times 10^{-2}$ &  $2.29 \times 10^{-2}$ &  $4.15 \times 10^{-3}$ &  $4.04 \times 10^{-4}$ & $9.40 \times 10^{-6}$ \\
        \hline
        $\| f - \Pi_+^N f \|_2$ &  $1.25 \times 10^{-1}$ &  $2.68 \times 10^{-2}$ &  $5.63 \times 10^{-3}$ &  $4.57 \times 10^{-4}$ & $1.12 \times 10^{-5}$ \\
        order &   &  $2.22$ &  $2.25$ &  $3.62$ & $5.35$ \\
        \hline
    \end{tabular}
    }
    \label{tab:projerror6}
\end{table}

\subsection{Application to the Boltzmann equation}
The initial value problem of the space homogeneous Boltzmann equation describes the evolution of non-negative distribution function $f(t,v)$ in three dimensions ($d=3$) as
\begin{equation}\label{eq:boltz}
    \left\{
    \begin{aligned}
        & \frac{\partial f}{\partial t} = Q(f,f),\\
        & f(t=0,v) = f_0(v),\qquad v \in \mathbb{R}^3,
    \end{aligned}
    \right.
\end{equation}
where the bilinear collision or interaction operator
\begin{equation}\label{eq:collision}
    Q(f,f)(v) = \int_{\mathbb{R}^3} \int_{\mathbb{S}^{2}} B(|v-v_*|,\cos \theta) \left(f(v_*') f(v') - f(v_*) f(v) \right) \mathrm{d} \sigma \mathrm{d}(v-v_*).
\end{equation}
The parameters in \cref{eq:collision} discribes the details of bilinear collision, where $(v,v_*)$ and $(v',v_*')$ are pre-collisional velocity pair and post-collisional velocity pair, respectively; $\sigma$ is the specular reflection direction and $\cos \theta = \frac{(v-v_*) \cdot \sigma}{|v-v_*|}$; $B(|v-v_*|,\cos \theta)$ is the collision kernel. For elastic collisions, the collision operator in \cref{eq:collision} satisfies
\begin{equation}
    \int_{\mathbb{R}^3} Q(f,f) \phi(v) \mathrm{d}v = 0, \qquad \phi(v) = 1, v_1,v_2,v_3, |v|^2,
\end{equation}
which implies the conservation of mass, momentum and energy for $f$ in \cref{eq:boltz}.

To discretize \cref{eq:boltz} by the Fourier spectral method (see \cite{Pareschi2000Fourier} or the review paper \cite{pareschi2014acta} and the references therein), we assume $f(t,x)$ has a compact support in a ball $B_{0}(R)$ and truncate the simulation domain into a larger bounded hypercube $[-L,L]^3$ where $L \geq  \frac{3+\sqrt{2}}{2}R$ to avoid aliasing error. Furthermore, we can periodize both the distribution function and the collision operator so that the right-hand side of the Boltzmann equation becomes
\begin{equation}\label{eq:truncol}
    Q^R(f,f)(v) = \int_{B_0(2R)} \int_{\mathbb{S}^{2}} B(|v-v_*|,\cos \theta) \left(f(v_*') f(v') - f(v_*) f(v) \right) \mathrm{d} \sigma \mathrm{d}(v-v_*).
\end{equation}

It should be noted that the operator $Q^R$ does not preserve all the collision invariants of the original operator.  In particular, the conservations of momentum and energy are lost. Considering the following expansion of the distribution function:
\begin{equation}\label{eq:ffourier}
    f(v) = \sum_{k \in \mathcal{N}^3} \hat{f}_k e^{i \frac{\pi}{L} k \cdot v}, \qquad \hat{f}_k = \frac{1}{(2 L)^3} \int_{[-L,L]^3} f(v) e^{i \frac{\pi}{L} k \cdot v} \mathrm{d}v,
\end{equation}
the Fourier coefficients of $Q^R(f,f)$ can be approximated by
\begin{equation}\label{eq:Qsum}
\hat{Q}_k = \sum_{(l+m) \bmod N =k} \left( \hat{B}(l,m)-\hat{B}(m,m) \right) \hat{f}_l \hat{f}_m,
\end{equation}
and the thus Fourier method for the truncated Boltzmann equation is
\begin{equation}\label{eq:fourierboltz}
    \frac{\mathrm{d} \hat{f}_k}{\mathrm{d} t} = \hat{Q}_k.
\end{equation}
The initial distribution function can be approximated by either interpolation or projection.

In our tests, we consider the Maxwell gas molecules by taking $B(|v-v_*|,\cos \theta) = \frac{1}{4 \pi}$, which implies $\hat{B}(l,m)$ in \cref{eq:Qsum} as
\begin{equation*}
    \hat{B}(l,m) = 32 \pi R^3 \frac{(\xi + \eta) \sin(\xi - \eta) - (\xi - \eta) \sin(\xi + \eta)}{2 \xi \eta (\xi + \eta)(\xi - \eta)},
\end{equation*}
where $\xi= \frac{|l+m|R\pi}{L},\eta= \frac{|l-m|R\pi}{L}$.

For a forward Euler method with uniform time step $\Delta t$, we denote the vector of the coefficients $\hat{f}_k$ at the $n$th time step as $\boldsymbol{f}^n$. Then the temporal discretization of \cref{eq:fourierboltz} becomes 
\begin{equation}
    \boldsymbol{f}^{n+1} = \boldsymbol{f}^n + \Delta t \boldsymbol{Q}(\boldsymbol{f}^n, \boldsymbol{f}^n),
\end{equation}
where $\boldsymbol{Q}$ is the vector of $\hat{Q}_k$. In general, this method preserves only mass conservation, and we will fix the momentum conservation, the energy conservation and the positivity by applying our positive projection after each time step, so that the numerical scheme becomes
\begin{equation}\label{eq:fourierposEuler}
    \left\{
    \begin{aligned}
    & \boldsymbol{f}^{n+1,*} = \boldsymbol{f}^n + \Delta t \boldsymbol{Q}(\boldsymbol{f}^n, \boldsymbol{f}^n),\\
    & f_N^{n+1} = \operatorname*{argmin}_{g \in \mathbb{S}_+^N} \|g - f_N^{n+1,*}\|_2^2 \quad \text{s.t.} \quad \boldsymbol{\rho}(g) = \boldsymbol{\rho}(f_N^n), \\
    & f_N^{0} = \Pi^N_+ f_0.
    \end{aligned}
    \right.
\end{equation}
For the purpose of comparison, we will also present the solution with moment conservation only, where $\mathbb{S}^N_+$ is the above scheme is replaced by $\mathbb{S}^N$.
Also, the scheme in \cref{eq:fourierposEuler} can be easily generalized to higher-order scheme by using Runge-Kutta methods. In our implementation, we use the third-order strong stability-preserving Runge-Kutta method \cite{sspreview2001} by applying projection to each stage.

In the following two examples, we take $R = 3$ and time step $\Delta t = 0.01$. The error of moments is computed by the $l^2$ error of the five conserved moments.

\paragraph{Example 1: BKW solution} The first example is the Bobylev-Krook-Wu (BKW) solution\cite{Bobylev1975,kw1977}, which is an exact solution of \cref{eq:boltz} as

\begin{equation}
    f(t,v) = \frac{1}{(2 \pi S)^{3/2}} \exp \left( - \frac{|v|^2}{2S} \right) \left( \frac{5S-3}{2S} + \frac{1-S}{2S^2} |v|^2 \right),
\end{equation}
where $S = 1 - 2 \exp(-t/6)/5$. 

\begin{table}[htbp]
    \centering
    \caption{Errors and properties of $f_N$ at $t = 0.1$.}
    \begin{tabular}{cccc}
        \hline
        N & 4 & 8 & 16 \\
        \hline
        \multicolumn{4}{c}{spectral method} \\
        \hline
        $\| \mathcal{I}^N f - f_N \|_2$ & $6.39\times10^{-2}$ & $3.03\times 10^{-3}$ & $2.10 \times 10^{-8}$ \\
        $\| \boldsymbol{\rho}(f) - \boldsymbol{\rho}(f_N) \|_2$ & $1.70\times10^{-3}$ & $2.66\times 10^{-3}$ & $1.06 \times 10^{-9}$ \\
        $\min_{k \in \mathcal{N}^3}(f_N(x_k))$ & $-2.55\times10^{-4}$ & $-1.43\times 10^{-4}$ & $-1.68 \times 10^{-10}$ \\
        \hline
        \multicolumn{4}{c}{moment projection} \\
        \hline
        $\| \mathcal{I}^N f - f_N \|_2$ & $6.39\times10^{-2}$ & $3.03 \times 10^{-3}$ & $2.10 \times 10^{-8}$ \\
        $\| \boldsymbol{\rho}(f) - \boldsymbol{\rho}(f_N) \|_2$ & $4.86\times10^{-15}$ & $2.49\times 10^{-14}$ & $2.70 \times 10^{-14}$ \\
        $\min_{k \in \mathcal{N}^3}(f_N(x_k))$ & $-2.55\times10^{-4}$ & $-1.43\times 10^{-4}$ & $-1.68 \times 10^{-10}$ \\
        \hline
        \multicolumn{4}{c}{positive and moment-preserving} \\
        \hline
        $\| \mathcal{I}^N f - f_N \|_2$ & $6.62\times 10^{-2}$ & $2.99 \times 10^{-3}$ & $2.10 \times 10^{-8}$ \\
        $\| \boldsymbol{\rho}(f) - \boldsymbol{\rho}(f_N) \|_2$ & $6.39\times10^{-16}$ & $4.01\times 10^{-15}$ & $1.26 \times 10^{-14}$ \\
        $\min_{k \in \mathcal{N}^3}(f_N(x_k))$ & $0$ & $0$ & $0$ \\
        \hline
    \end{tabular}
    \label{tab:BKW1}
\end{table}

It can be seen in \cref{tab:BKW1} that all three methods exhibit spectral convergence at $t = 0.1$. Meanwhile, although both moment-preserving methods preserve the conservation of all moments (up to machine accuracy), our new positive and moment-preserving Fourier method is the only one that preserves positivity.

\begin{figure}[htbp]
	\centering
	\subfigure[$\| \mathcal{I}^N f - f_N \|_2$.]
	{\includegraphics[width=.32\textwidth]{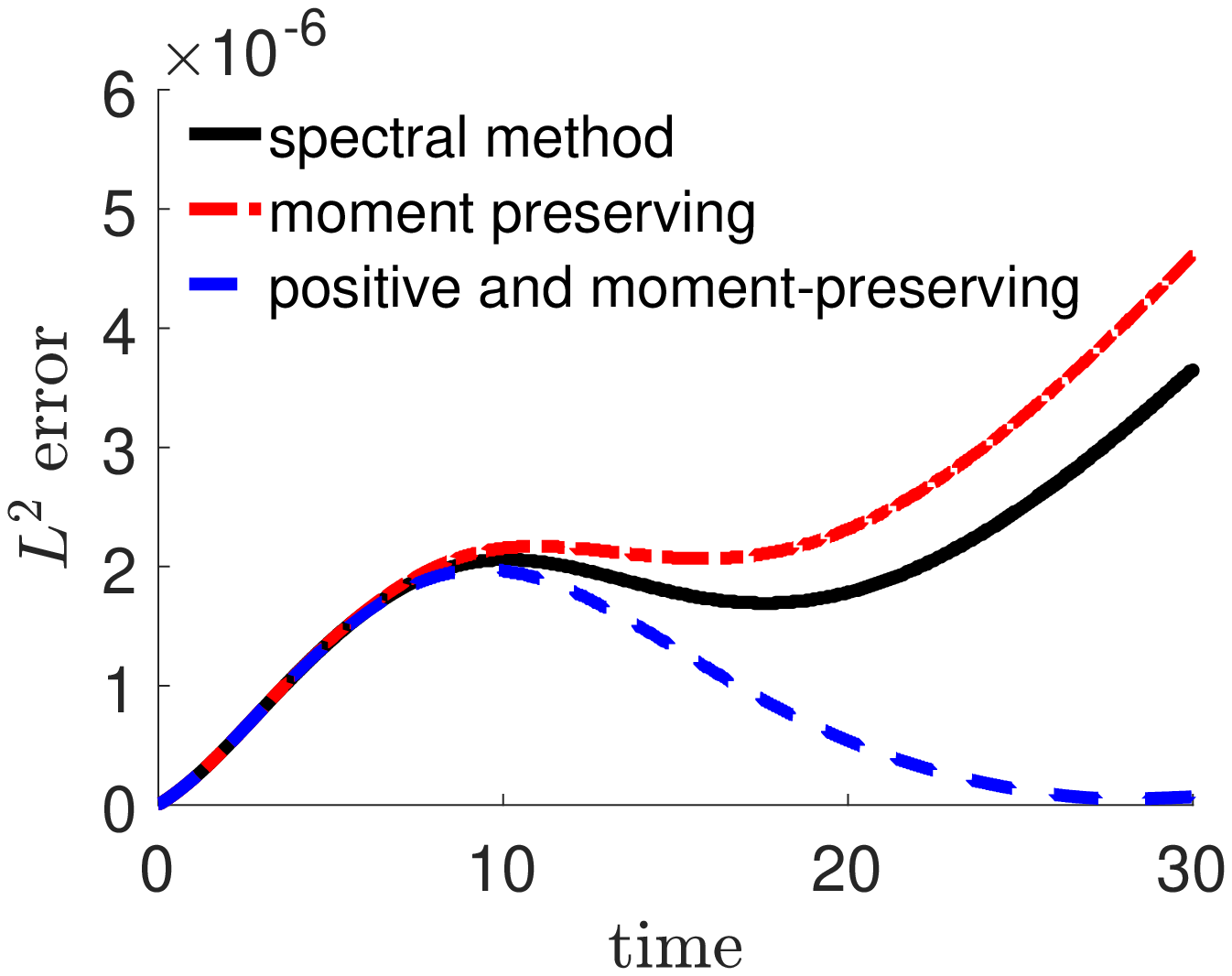}}
    \subfigure[$\| \boldsymbol{\rho}(f) - \boldsymbol{\rho}(f_N) \|_2$.]
	{\includegraphics[width=.32\textwidth]{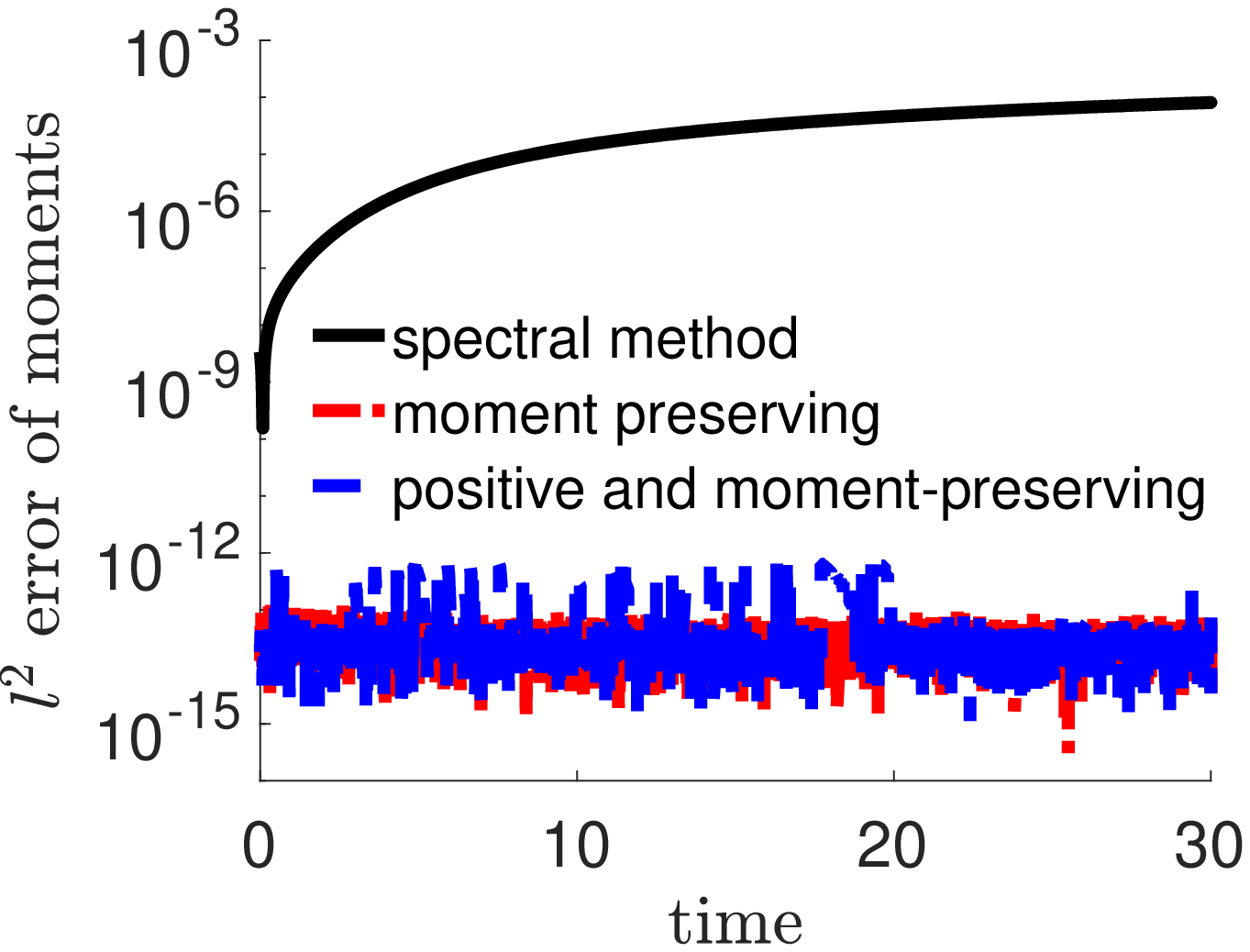}}
    \subfigure[$\min_{k \in \mathcal{N}^3}(f_N(x_k))$.]
	{\includegraphics[width=.32\textwidth]{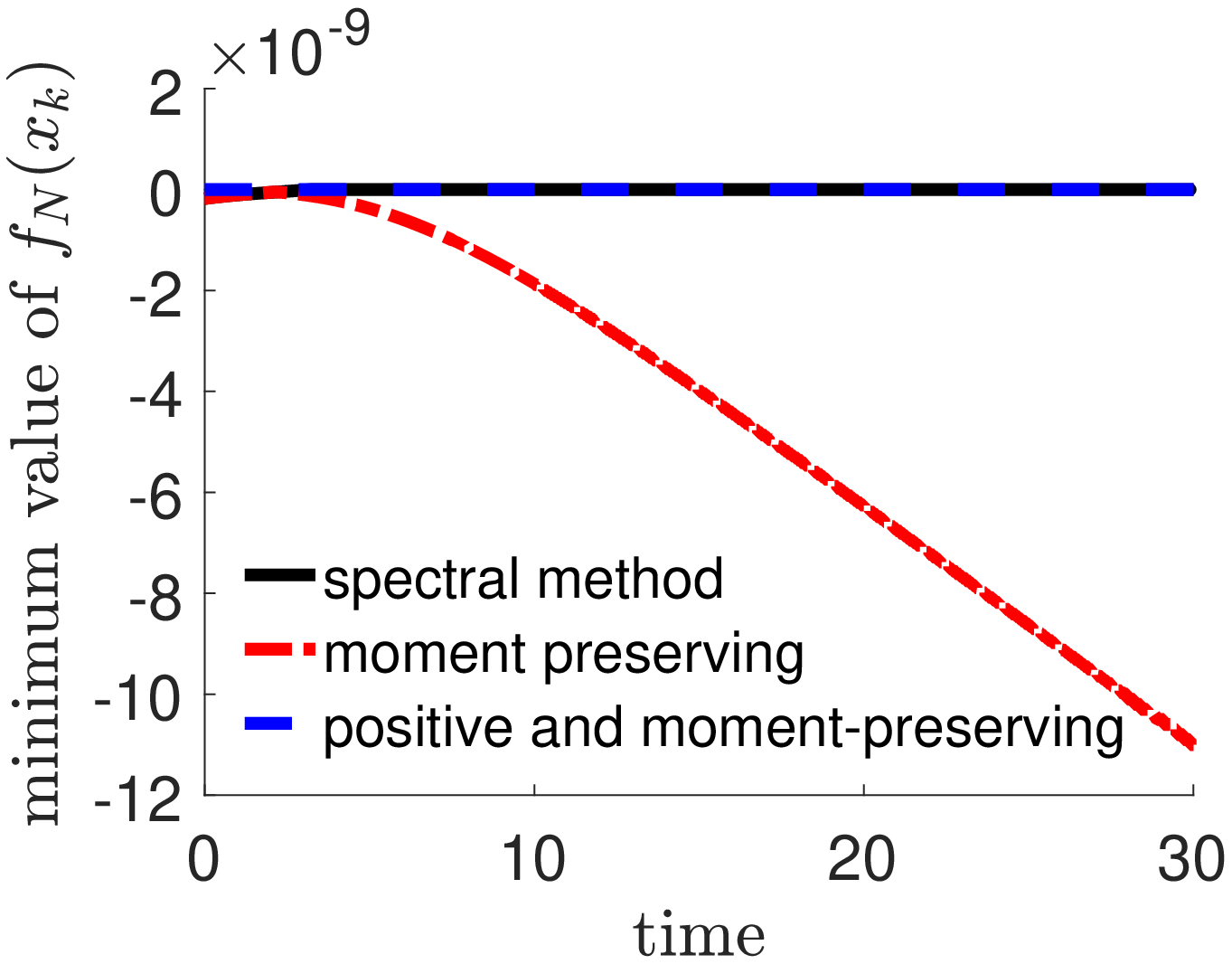}}
	
        \caption{Example of the Boltzmann equation for BKW solution. Time evolution of the $L^2$ error $\| \mathcal{I}^N f - f_N \|_2$, $l^2$ error of moments $\| \boldsymbol{\rho}(f) - \boldsymbol{\rho}(f_N) \|_2$ and minimum value $\min_{k \in \mathcal{N}^3}(f_N(x_k))$ for $N=16$.
        \label{fig:bkw3d}
        }
\end{figure}

We are also interested in the long-time performance of all three methods in \cref{tab:BKW1}. Therefore, we simulate the BKW problem to a longer time $t=30$ for $N=16$, and we plot the result in \cref{fig:bkw3d}. It can be observed that the positive and moment-preserving method achieves better accuracy for a relatively longer simulation.

\paragraph{Example 2: discontinuous initial value}
This example adopts a discontinuous initial value \cite{cai2018entropic} as
\begin{displaymath}
    f_0(v) = \left\{
    \begin{aligned}
        \frac{\rho_1}{(2 \pi T_1)^{3/2}} \exp \left( - \frac{|v|^2}{2 T_1}\right), \ \text{ if } v_1 >0, \\
        \frac{\rho_2}{(2 \pi T_2)^{3/2}} \exp \left( - \frac{|v|^2}{2 T_2}\right), \ \text{ if } v_1 <0, 
    \end{aligned}
    \right.
\end{displaymath}
where $\rho_1 = 6/5$, $\rho_2 = 4/5$, $T_1 = 2/3$ and $T_2 = 3/2$. Numerical solutions for $N=16$ on the line $y=z=0$ at $t=0.5$ are plotted in \cref{fig:discontinuous}.

\begin{figure}[htb]
	\centering
	\includegraphics[width=.44\textwidth]{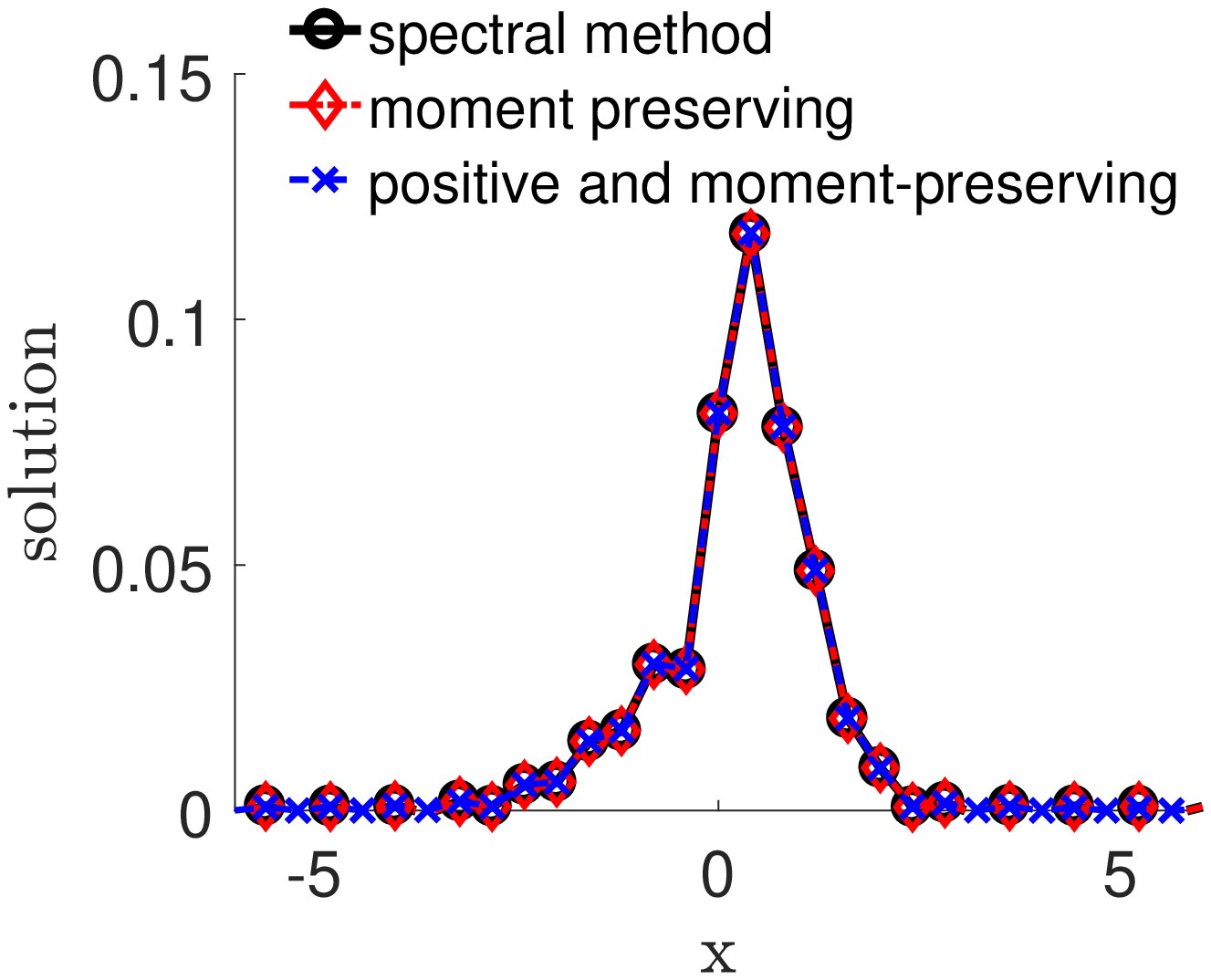}
    \includegraphics[width=.44\textwidth]{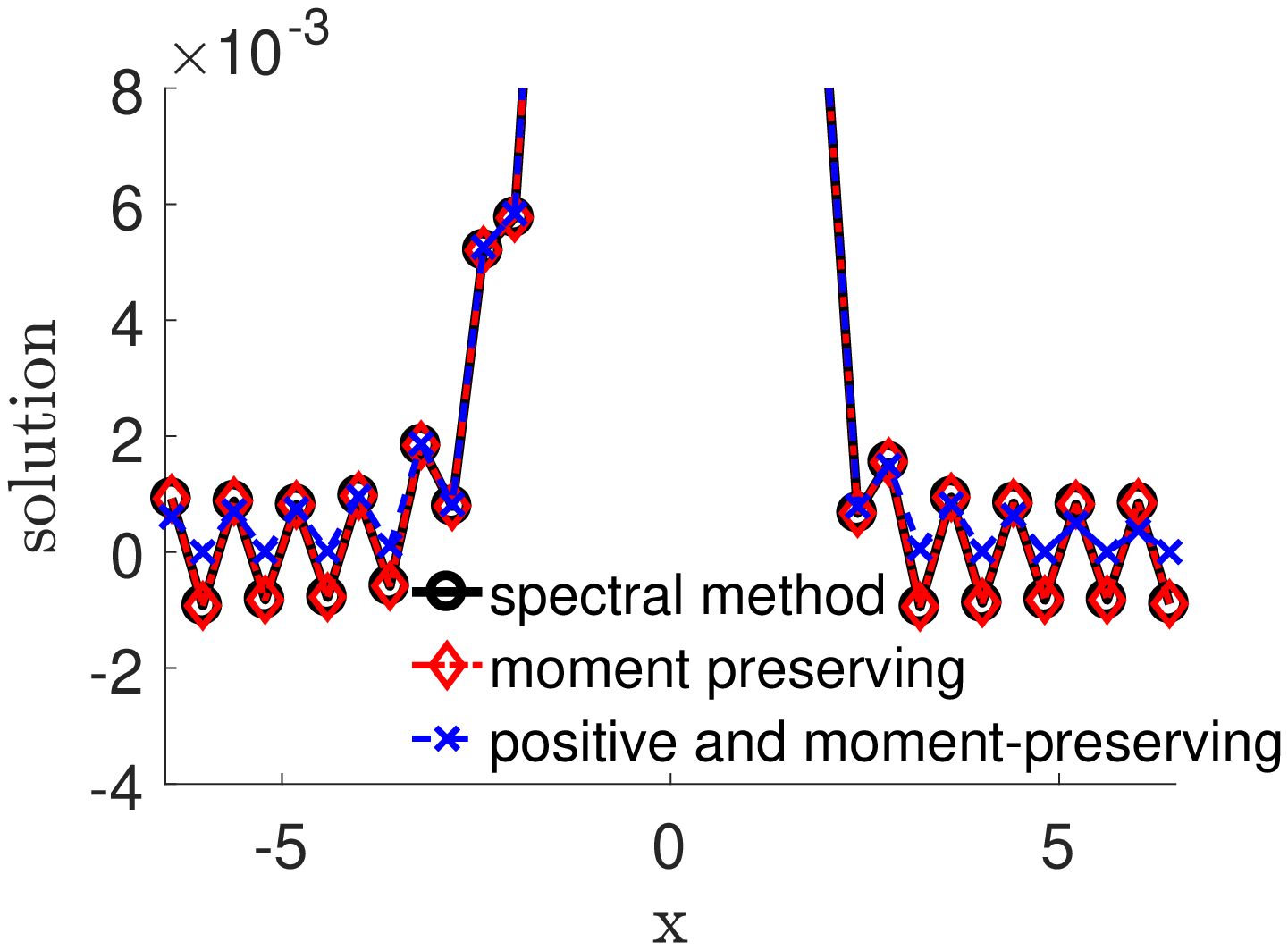}
        \caption{Example of the Boltzmann equation for discontinuous initial value. $N=16$ and $t = 0.5$.
        \label{fig:discontinuous}
        }
\end{figure}

Due to the discontinuity occurring in the initial value, we observe the Gibbs phenomenon for all three methods in \cref{fig:discontinuous}. Nevertheless, the solution of positive and moment-preserving method is always non-negative, and the magnitude of its oscillation is smaller than that of the other two methods.

\section{Conclusions}\label{sec:conclusions}
This work focuses on maintaining the positivity and the conservation of moments when applying the Fourier spectral method. By introducing an optimization problem in the space of trigonometric polynomials, a new projection is constructed such that all moments are conserved and the projected trigonometric polynomial is non-negative on all collocation points. We analyze the accuracy of the new projection, and prove the spectral accuracy is maintained. Moreover, an efficient and practical Newton's method is proposed to solve the associated optimization problem, which is proved to enjoy quadratic convergence. The spectral accuracy of the new Fourier spectral method is further validated by numerical examples including applications to the Boltzmann equation.

Regarding the generality of this new projection, future works include the extension to other equations where positivity is an important issue in numerical implementation, e.g., the Fokker-Planck equation or the Cahn-Hilliard equation with a logarithmic potential. The idea of constructing new projection may also be applied to other spectral methods.

\appendix
\section{Proof of \cref{lemma:poly}}\label{appendix:polynomial}
With the help of triangle inequality, it suffices to estimate the $H^1$ semi-norm and the projection error of monomials, which are given in \cref{col:H1semi,col:projmono}, respectively. They are based on the computation and estimate of monomials $x^{\alpha} = x_1^{\alpha_1} x_2^{\alpha_2} \cdots x_d^{\alpha_d}$ shown in the following lemma.

\begin{lemma}[Fourier coefficients of monomials]\label{lemma:polyfourier} For a given monomial $x^{\alpha}$, its Fourier coefficients $(\widehat{x^{\alpha}})_k$ satisfy
\begin{equation}\label{eq:monocoefest}
    | (\widehat{x^{\alpha}})_k | \leq C(|\alpha|,d) \prod_{j=1}^d \left( \delta_{0,k_j} + \frac{1 - \delta_{0,k_j}}{|k_j|} \right),\qquad \forall k = (k_1,...,k_d)^{\intercal},
\end{equation}
where $\frac{1 - \delta_{0,k_j}}{|k_j|}:=0$ when $k_j=0$, and the constant $C(|\alpha|,d)= \frac{2^d \pi^{|\alpha|} |\alpha|^d}{d^d}$.

\end{lemma}
\begin{proof}
    A direct computation yields
\begin{equation}\label{eq:fourierpoly}
    (\widehat{x^{\alpha}})_k = \frac{1}{(2\pi)^d} \int_{[-\pi, \pi]^d} x^{\alpha} e^{-i k \cdot x }\mathrm{d}x = \prod_{j=1}^d F(k_j, \alpha_j), 
\end{equation}
where $F(m,n)$ for two integers $m$ and $n \geq 0$ is the one-dimensional Fourier transform of monomial
\begin{equation}\label{eq:fourierpoly2}
    F(m, n) = \frac{1}{2\pi} \int_{-\pi}^{\pi} t^{n} e^{-i m t} \mathrm{d}t.
\end{equation}
When $m=0$, $F(0, n) = \frac{\left((-1)^{n }+1\right) \pi^{n}}{2(n +1)}$. When $m \neq 0$ and $n = 0$, $F(m, 0) = 0$. When $m \neq 0$ and $n > 0$, we can apply integration by parts and get
\begin{equation}\label{eq:fourierpolyoddeven}
    F(m, n) = \left\{ \begin{array}{cc}
       i(-1)^m\frac{\pi^{n-1}}{m} - \frac{n(n-1)}{m^2} F(m,n-2),  & n \text{ is odd}, \\
       (-1)^m\frac{n\pi^{(n-2)}}{m^2} - \frac{n(n-1)}{m^2} F(m,n-2), & n \text{ is even},
    \end{array}
    \right.
\end{equation}
where we specify $F(m,-1)=0$ in the above equality. By the definition in \cref{eq:fourierpoly2}, it holds that when $n>1$,
\begin{equation*}
    |F(m,n-2)| \leq \frac{1}{2\pi} \int_{-\pi}^{\pi} |t|^{n-2} \mathrm{d}t = \frac{\pi^{n-2}}{n-1}.
\end{equation*}
Plugging the above inequality into \cref{eq:fourierpolyoddeven} yields 
$|F(m, n)| \leq \frac{2 n \pi^{n-1}}{|m|}$ for all $m \neq 0$. On the other hand, it is obvious that $|F(0, n)| = \left| \frac{\left((-1)^{n }+1\right) \pi^{n}}{2(n +1)} \right| \leq \pi^n$. Combining the above different cases of $|F(m, n)|$, we get an estimate for all $m$ and $n$ that
\begin{equation}
    |F(m, n)| \leq 2 n \pi^n \left( \delta_{0,m} + \frac{1-\delta_{0,m}}{|m|}\right).
\end{equation}
By plugging the above estimate into \cref{eq:fourierpoly}, we obtain the desired \cref{eq:monocoefest} where the constant $C(|\alpha|,d)$ comes from $\prod_{j=1}^d 2 \alpha_j \pi^{\alpha_j} = 2^d \pi^{|\alpha|} \prod_{j=1}^d \alpha_j \leq \frac{2^d \pi^{|\alpha|} |\alpha|^d}{d^d}$.
\end{proof}

Based on this lemma, two corollaries will be shown, which are related to the $H^1$ semi-norm and the projection error of monomials.

\begin{corollary}[$H^1$ semi-norm of monomials] \label{col:H1semi} The $H^1$ semi-norm of $\mathcal{P}^N x^{\alpha}$ satisfies
    \begin{equation}\label{eq:pnh1}
        | \mathcal{P}^N x^{\alpha} |_{H^1} \leq C(|\alpha|,d) \sqrt{N},
    \end{equation}
    where $C(|\alpha|,d) = \sqrt{2d \left( \frac{2^d \pi^{|\alpha|} |\alpha|^d}{d^d} \right)^2 \left( \frac{\pi^2}{3} + 1 \right)^{d-1}}$ is a constant.
\end{corollary}
\begin{proof}
    By the definition of $H^1$ semi-norm and \cref{lemma:polyfourier},
\begin{align*}
    | \mathcal{P}^N x^{\alpha} |_{H^1}^2 & = \sum_{k \in \mathcal{N}^d} \sum_{l=1}^d |k_l|^2 | (\widehat{x^{\alpha}})_k |^2 \\
    & \leq \left( \frac{2^d \pi^{|\alpha|} |\alpha|^d}{d^d} \right)^2 \sum_{l=1}^d \sum_{k \in \mathcal{N}^d} |k_l|^2 \prod_{j=1}^d \left( \delta_{0,k_j} + \frac{1 - \delta_{0,k_j}}{|k_j|} \right)^2 \\
    & = \left( \frac{2^d \pi^{|\alpha|} |\alpha|^d}{d^d} \right)^2 \sum_{l=1}^d \sum_{k_l=-N}^N \left( |k_l| \delta_{0,k_l} + 1 - \delta_{0,k_l} \right)^2 \prod_{\substack{j=1 \\ j \neq l}}^d \sum_{k_j=-N}^N \left( \delta_{0,k_j} + \frac{1 - \delta_{0,k_j}}{|k_j|} \right)^2 \\
    & = \left( \frac{2^d \pi^{|\alpha|} |\alpha|^d}{d^d} \right)^2 \sum_{l=1}^d (2N) \prod_{\substack{j=1 \\ j \neq l}}^d \left( 2\sum_{k_j=1}^N \frac{1}{|k_j|^2} + 1  \right) \\
    & \leq \left( \frac{2^d \pi^{|\alpha|} |\alpha|^d}{d^d} \right)^2 d (2N) \left( \frac{\pi^2}{3} + 1 \right)^{d-1}.
\end{align*}
\end{proof}

\begin{corollary}[Projection error of monomials] \label{col:projmono} For a monomial $x^{\alpha}$,
\begin{equation}\label{eq:polyprojerr}
    \| \mathcal{P}^N x^{\alpha} - x^{\alpha} \|_2 \leq C(|\alpha|,d) N^{-1/2},
\end{equation}
where constant $C(|\alpha|,d) = \sqrt{2 d \left( \frac{2^d \pi^{|\alpha|} |\alpha|^d}{d^d} \right)^2 \left( \frac{\pi^2}{3} + 1 \right)^{d-1}}$.
\end{corollary}
\begin{proof}
By Parseval's theorem and \cref{lemma:polyfourier},
\begin{align*}
    \| \mathcal{P}^N x^{\alpha} - x^{\alpha} \|_2^2 
    & \leq \sum_{l=1}^d \sum_{k_1=-\infty}^{\infty} \cdots \sum_{k_{l-1}=-\infty}^{\infty} \sum_{|k_l|>N} \sum_{k_{l+1}=-\infty}^{\infty} \cdots \sum_{k_d=-\infty}^{\infty} |(\widehat{x^{\alpha}})_k|^2 \\
    & \leq \left( \frac{2^d \pi^{|\alpha|} |\alpha|^d}{d^d} \right)^2  \sum_{l=1}^d \sum_{|k_l|>N} \frac{1}{|k_l|^2} \prod_{\substack{j=1 \\ j \neq l}}^d \sum_{k_j=-\infty}^{\infty} \left( \delta_{0,k_j} + \frac{1 - \delta_{0,k_j}}{|k_j|} \right)^2 \\
    & = \left( \frac{2^d \pi^{|\alpha|} |\alpha|^d}{d^d} \right)^2 2 d \sum_{n > N} \frac{1}{n^2} \left( \frac{\pi^2}{3} + 1 \right)^{d-1}.
\end{align*}
On the other hand, $\sum_{n > N} \frac{1}{n^2} = \psi^{(1)}(N+1)$, where $\psi^{(m)}(z)$ denotes the polygamma function \cite[Section 6.4]{abramowitz1964handbook}. It is shown in \cite{QI2010polygammabound} that for $z > 0$,
\begin{equation*}
    (-1)^{(m+1)}\psi^{(m)}(z) < \frac{(m-1)!}{z^m} + \frac{m!}{z^{m+1}}.
\end{equation*}
 Therefore,
\begin{equation*}
    N \sum_{n > N} \frac{1}{n^2} = N \psi^{(1)}(N+1) \leq \frac{N}{N+1} + \frac{N}{(N+1)^2} \leq 1,
\end{equation*}
and \cref{eq:polyprojerr} is proved.
\end{proof}

\section{Proof of \cref{thm: ssn}}
\label{sec: convergence_SSN}
We first note that the function $\nabla \Phi(\boldsymbol{\cdot})$ is Lipschitz continuous on ${\rm Ran}(\mathcal{M})$, as $\Pi_{+}(\cdot)$ is Lipschitz continuous with modulus 1 \cite{rockafellar1976monotone}. According to Rademacher’s theorem \cite[Section 9.J]{rockafellar2009variational}, $\nabla \Phi(\boldsymbol{\cdot})$ is almost everywhere Fr$\acute{\rm e}$chet-differentiable in ${\rm Ran}(\mathcal{M})$. Let $\boldsymbol{\lambda}\in {\rm Ran}(\mathcal{M})$ be any given point. Define the following operator
\begin{align*}
\hat{\partial}^2 \Phi (\boldsymbol{\lambda}):=- \frac{1}{2} \mathcal{M} \partial \Pi_+ (\boldsymbol{f}_N + \frac{1}{2} \mathcal{M}^{\intercal} \boldsymbol{\lambda} ) \mathcal{M}^{\intercal},
\end{align*}
where $\partial \Pi_+(\boldsymbol{f}_N + \frac{1}{2} \mathcal{M}^{\intercal} \boldsymbol{\lambda} )$ is the Clarke subdifferential \cite{clarke1990optimization} of the Lipschitz continuous function $\Pi_+(\boldsymbol{\cdot})$ at $\boldsymbol{f}_N + \frac{1}{2} \mathcal{M}^{\intercal} \boldsymbol{\lambda}$, defined as the following set of diagonal matrices:
\begin{equation}\label{eq:partialpi}
    \partial \Pi_+(\boldsymbol{w}) = \left\{ \text{Diag}(\boldsymbol{u}) \,\middle\vert\ 
    \begin{array}{ll}
        u_k=1, & w_k > 0, \\
        u_k\in {[0,1]}, & w_k = 0, \\
        u_k=0, & w_k < 0,
    \end{array}
    \right\},\quad \mbox{for all }\boldsymbol{w}\in \mathbb{R}^{(2N+1)^d}.
\end{equation}
From \cite{hiriart1984generalized}, we have that 
\begin{align*}
\partial^2 \Phi (\boldsymbol{\lambda}) \boldsymbol{d} = \hat{\partial}^2 \Phi (\boldsymbol{\lambda}) \boldsymbol{d},\quad \mbox{for all }\boldsymbol{d}\in \mathbb{R}^M,
\end{align*}
where $\partial^2 \Phi (\boldsymbol{\lambda})$ denotes the Clarke subdifferential of $\nabla \Phi(\cdot)$ at $\boldsymbol{\lambda}$. Moreover, it can be seen that the elements in $\hat{\partial}^2 \Phi (\boldsymbol{\lambda})$ are all symmetric and negative semidefinite. Furthermore, $\nabla \Phi(\boldsymbol{\cdot})$ is strongly semismooth \cite{sun2008lowner} with respect to $\hat{\partial}^2 \Phi (\boldsymbol{\lambda})$. 

Based on the above discussions, we can design a semismooth Newton method to solve the nonsmooth equation \cref{eq:dualopt}. The semismooth Newton method is a generalized version of the classical Newton method, where the main modification is to replace the Hessian matrix by a generalized Hessian operator. In our case, any element in $\hat{\partial}^2 \Phi (\boldsymbol{\lambda})$ can be seen as a generalized Hessian of $\Phi(\boldsymbol{\cdot})$. In particular, we can pick 
\begin{align*}
- \frac{1}{2} \mathcal{M} \text{Diag}(\boldsymbol{u}) \mathcal{M}^{\intercal} \in \hat{\partial}^2 \Phi (\boldsymbol{\lambda}),
\end{align*}
where $\boldsymbol{u}\in \mathbb{R}^{(2N+1)^d}$ is defined as: for each $k$, $u_k = 1$ if $(\boldsymbol{f}_N + \frac{1}{2} \mathcal{M}^{\intercal} \boldsymbol{\lambda} )_k>0$ and $u_k=0$ otherwise.

The first part of the convergence results in \cref{thm: ssn} then follows from \cite[Theorem 3.4]{zhao2010newton}. In order to ensure the quadratic convergence rate in the second part, according to \cite[Theorem 3.5]{zhao2010newton}, we only need the constraint nondegeneracy assumption
\begin{align}
\mathcal{M} \ {\rm lin}(\mathcal{T}_{\mathbb{R}^{(2N+1)^d}_+}(\boldsymbol{g}^*))={\rm Ran}(\mathcal{M}), \label{eq:constraint_nondegeneracy}
\end{align}
where $\mathcal{T}_{\mathbb{R}^{(2N+1)^d}_+}(\boldsymbol{g}^*)$ denotes the tangent cone of $\mathbb{R}^{(2N+1)^d}_+$ at $\boldsymbol{g}^*$, and ${\rm lin}(\mathcal{T}_{\mathbb{R}^{(2N+1)^d}_+}(\boldsymbol{g}^*))$ denotes the lineality space of $\mathcal{T}_{\mathbb{R}^{(2N+1)^d}_+}(\boldsymbol{g}^*)$, that is the greatest linear subspace contained in $\mathcal{T}_{\mathbb{R}^{(2N+1)^d}_+}(\boldsymbol{g}^*)$. By noting the fact that
\begin{align*}
\mathcal{T}_{\mathbb{R}^{(2N+1)^d}_+}(\boldsymbol{g}^*) = \left\{
\boldsymbol{g} \in \mathbb{R}^{(2N+1)^d} \middle\vert g_k \geq 0,\ k \in \mathcal{N}^d \backslash {\cal K}
\right\},
\end{align*}
where ${\cal K}$ is defined in the theorem, we have that
\begin{align*}
{\rm lin}(\mathcal{T}_{\mathbb{R}^{(2N+1)^d}_+}(\boldsymbol{g}^*)) = \left\{
\boldsymbol{g} \in \mathbb{R}^{(2N+1)^d} \middle\vert g_k = 0,\ k \in \mathcal{N}^d \backslash {\cal K}
\right\}.
\end{align*}
Therefore, we can see that the condition \cref{eq:constraint_nondegeneracy} is equivalent to \cref{eq:span_condition} in \cref{thm: ssn}. This completes the proof.

\end{document}